\documentclass{article}

\usepackage{graphicx} 
\usepackage[a4paper, total={6in, 8in}]{geometry}
\usepackage{fancyhdr}
\usepackage{pdfpages}
\usepackage{hyperref}
\usepackage[hypcap=true]{caption}
\hypersetup{colorlinks = true, linkcolor = purple, citecolor= blue}
\usepackage[toc,page]{appendix}
\usepackage{dsfont}
\usepackage[utf8]{inputenc}
\usepackage{bbm}
\usepackage{amsmath}
\usepackage[normalem]{ulem}
\usepackage{amsfonts}
\usepackage{amssymb}
\usepackage{amsthm} 
\usepackage{subcaption}
\usepackage[toc,page]{appendix}
\newtheorem*{theorem*}{Theorem}
\newtheorem{theorem}{Theorem}[section]
\newtheorem{corollary}[theorem]{Corollary}
\newtheorem{lemma}[theorem]{Lemma}
\newtheorem{definition}[theorem]{Definition}
\newtheorem{remark}[theorem]{Remark}
\newtheorem*{remark*}{Remark}
\newtheorem{proposition}[theorem]{Proposition}
\usepackage{enumitem}

\usepackage[
  backend=biber,
  style=numeric,
  sortcites=true,
  defernumbers=true,
  giveninits=true,   
  eprint=false,
  url=false,         
  doi=false,         
  isbn=false         
]{biblatex}
\usepackage[section]{placeins}
\renewcommand{\epsilon}{\varepsilon}
\usepackage{float}

\usepackage{todonotes}

\DeclareUnicodeCharacter{2217}{*}

\title{Non Null-Controllability Properties of the Grushin-Like Heat Equation on 2D-Manifolds}
\author{Roman Vanlaere \footnote{CEREMADE, Université Paris-Dauphine PSL, CNRS UMR 7534, 75016 Paris, France, roman.vanlaere@dauphine.psl.eu}}

\date{December, 2025}

\begin{document}

\maketitle

\begin{abstract}
     We study the internal non null-controllability properties of the heat equation on 2-dimensional almost-Riemannian manifolds with an interior singularity, and under the assumption that the closure of the control zone does not contain the whole singularity. We show that if locally, around the singularity, the sub-Riemannian metric can be written in a Grushin form, or equivalently the sub-Laplacian writes as a generalized Grushin operator, then, achieving null-controllability requires at least a minimal amount of time. As locally the manifold looks like a rectangular domain, we consequently focus ourselves on the non null-controllability properties of the generalized Grushin-like heat equation on various Euclidean domains.  
\end{abstract}

\tableofcontents

\section{Introduction}

Over the past decade, there has been significant attention devoted to exploring the controllability properties of evolution equations associated with degenerate elliptic operators. These are parabolic operators whose symbol can vanish. For the heat equation, this interest has been spurred by positive answers to the question of controllability of the heat equation associated to the Laplace-Beltrami operator on Riemannian manifolds (see \textit{e.g.} \cite{Fattorini1971-sq, fursikov1996controllability, Lebeau1995-jy, miller2005unique}). A pioneer work for the sub-elliptic heat equation, more precisely the Grushin equation, is very recent (see \textit{e.g.} \cite{beauchard2014null}), and shows that in this case, at least a minimum amount of time is required. As for the sub-elliptic wave equation on compact manifold, it has been very recently shown in \cite{letrouit2023subelliptic} that it is never null-controllable. However, this question for the sub-elliptic heat in a general context still lacks of answers. \\

In the present paper, we therefore study the null-controllability properties of the heat equation
\begin{align*}
    (\partial_t - \Delta)f = \mathbf{1}_\omega u,
\end{align*}
where $\Delta$ denotes a sub-Laplacian on a two-dimensional almost-Riemannian manifold, under the assumption that, locally, $\Delta$ can be written as a generalized Grushin-type operator.

\subsection{The sub-Riemannian setting}\label{section setting}

We set here the geometric settings of the present paper. For details on sub-Riemannian geometry, see for instance \cite{Jean2014-dz, Agrachev2019-nm}.\\

Let $\mathcal{M}$ be an orientable two-dimensional manifold, with or without boundary, and let $\mathcal{Z} \subset \mathcal{M}$ be a connected embedded one-dimensional submanifold. We will assume throughout this work that $\mathcal{M}$ and $\mathcal{Z}$ satisfy the following. 

\begin{enumerate}[label=\textbf{H\arabic*},ref=\text{H\arabic*}, start = 0]
    \item \label{H0} Let $\gamma \in  \mathbb{N}^*$. $\mathcal{M}$ is a complete almost-Riemannian manifold (see \cite[Definition 9.1]{Agrachev2019-nm}), of step  $1$ on $\mathcal{M} \setminus \mathcal{Z}$, and $\gamma + 1$ on $\mathcal{Z}$ (see \cite[Definition 3.1]{Agrachev2019-nm}). In particular, $\mathcal{Z}$ is non-characteristic (in the sense of \cite[Definition 2.3]{Franceschi2020-gs}). The injectivity radius from $\mathcal{Z}$, that we denote by $\text{inj}(\mathcal{Z})$, is bounded below by some $C > 0$, and $\partial \mathcal{Z} \subset \partial\mathcal{M}$. 
\end{enumerate}

The condition on the step of the structure implies the Hörmander condition, also known as bracket-generating condition. Namely, the sub-Riemannian structure on $\mathcal{M}$ is locally generated by a family of smooth vector fields such that, at each point, their iterated Lie brackets span the tangent space to $\mathcal{M}$. In our case, we need $\gamma$ iterations of Lie brackets on $\mathcal{Z}$, while none are needed on $\mathcal{M} \setminus \mathcal{Z}$. Given such a family $\mathcal{D}$ of vector fields, we can consider a metric $G$ which is defined to make $\mathcal{D}$ an orthonormal frame on $\mathcal{M} \setminus \mathcal{Z}$. The metric $G$ is Riemannian almost everywhere (except on $\mathcal{Z}$). Thanks to hypothesis \ref{H0}, Chow-Raschevskii theorem applies, and the notion of (sub-Riemannian) distance associated to the metric $G$ is well-defined on $\mathcal{M}$. We denote it by $d_{\operatorname{sR}}$, and $(M,d_{\operatorname{sR}})$ is a metric space. We also write, for $p \in \mathcal{M}$, 
\begin{align*}
    \delta(p) = \inf\{d_{\operatorname{sR}}(p,q), q \in \mathcal{Z}\}.
\end{align*}

Observe that due to \ref{H0}, any open connected neighborhood $\mathcal{O}$ of $\mathcal{Z}$ must satisfy $\mathcal{O} \setminus \mathcal{Z} = \mathcal{O}^- \sqcup \mathcal{O}^+$. We make the following complementary assumption to \ref{H0}.

\begin{enumerate}[label=\textbf{H0'},ref=\text{H0'}]
    \item \label{H0'} There exists $R \in (0,\operatorname{inj}(\mathcal{Z})]$ such that, letting $\mathcal{O} = \{p \in \mathcal{M},\delta(p) < R\}$, $\delta$ is smooth in $\mathcal{O}^{\pm} \cup \mathcal{Z}$. 
\end{enumerate}

\begin{remark}
The assumption on the injectivity radius in \ref{H0} is automatically verified whenever $\mathcal{Z}$ is compact and without boundary, and $\mathcal{Z}$ is always compact when $\mathcal{M}$ is compact, since $\mathcal Z$ is a closed embedded smooth submanifold of $\mathcal{M}$. The smoothness of $\delta$ in \ref{H0'} is automatically verified if $\partial\mathcal{M} = \emptyset$ (note that in this case, $\partial \mathcal{Z} = \emptyset)$. In the case with boundary, assumption \ref{H0'} is sufficient to ensure that our arguments work under \ref{H1ii}, but can be relaxed under \ref{H1i} (see the assumptions below and Lemma \ref{lemma: bndry tub neighb}). Moreover, the assumption $\partial \mathcal{Z} \subset \partial\mathcal{M}$ is to ensure the right geometrical setting under which we can work in the case $\gamma > 1$.
\end{remark}

The assumption \ref{H0} ensures the existence of a double-sided tubular neighborhood around $\mathcal{Z}$, that we denote by $\mathcal{Z}_L \subset \{ p \in \mathcal{M}, \delta(p) < L\}$, for some $0 < L < \operatorname{inj}(\mathcal{Z})$, and $\mathcal{Z}_L \simeq (-L,L) \times \mathcal{Z}$. Thanks to the smoothness assumption on $\delta$, we actually have that $\mathcal{Z}_L = \{ p \in \mathcal{M}, \delta(p) < L\}$ for $L \in (0,R)$ (see Lemma \ref{lemma: bndry tub neighb}). Whenever $\mathcal{Z}$ is compact and without boundary, such a neighborhood always exists under the sole assumption that $\mathcal{Z}$ is non-characteristic, and the set equality $\mathcal{Z}_L = \{ p \in \mathcal{M}, \delta(p) < L\}$ is verified (see e.g. \cite[Proposition 3.1]{Franceschi2020-gs}). Whenever $\mathcal{Z}$ is non-compact or has nonempty boundary, its existence is ensured by the assumption on the injectivity radius (see e.g.  \cite[Theorem 3.7]{rossi2022relative} combined with \cite[Proposition 3.1]{Franceschi2020-gs}). \\

We now introduce $\omega$ to be an open set of $\mathcal{M}$ that satisfies one of the following. 

\begin{enumerate}[label=\empty,ref=\text{H\arabic*}, start = 1] 
     \item \label{H1} 
     \begin{enumerate}[label = \textbf{H1-loc}, ref = \text{H1-loc}]
     \item \label{H1i} $\mathcal{M} \setminus \overline{\omega}$ contains a point $p \in \mathcal{Z}$. 
     \end{enumerate}
    \begin{enumerate}[label = \textbf{H1-glob}, ref = \text{H1-glob}]
    \item \label{H1ii} $d_{\operatorname{sR}}(\omega, \mathcal{Z}) > 0$.
    \end{enumerate}
\end{enumerate}

We observe that there exists an open set $\mathcal{U} \subset \mathcal{M} \setminus \overline{\omega}$, that is either
\begin{enumerate}[label=(\roman*),ref=\text{H\arabic*}, start = 1]
    \item a neighborhood of the point $p$ appearing in \ref{H1i},
    \item a neighborhood of $\mathcal{Z}$ if we are under \ref{H1ii},
\end{enumerate}
such that $\mathcal{U}$ is diffeomorphic to $(-L,L)\times \Omega_y$, where $\Omega_y = \mathbb{S}^1, \ \mathbb{R}$, or a bounded interval (see Figure \ref{fig1}). Indeed, in the setting of \ref{H1ii}, it is sufficient to choose $\mathcal{U} = \mathcal{Z}_L$, with $0 < L < \min(d_{\operatorname{sR}}(\omega, \mathcal{Z}), \operatorname{inj}(\mathcal{Z}))$. In the setting of \ref{H1i}, we can choose $\mathcal{U}$ to be a neighborhood of such a point $p$, such that it is diffeomorphic to $(-L,L) \times \mathcal{W}$, for some $L > 0$, with $\mathcal{W} \subset \mathcal{Z} \setminus \omega$ a relatively compact neighborhood of $p$ in $\mathcal{Z}$. In either case, since $\mathcal{Z}$ is one-dimensional and connected, $\mathcal{U}$ is always diffeomorphic to a set of the form $(-L,L) \times \Omega_y$. From now on, the notation $\mathcal{U}$ will designate a tubular neighborhood constructed as above under one of the setting outlined in \ref{H1}. \\ 

\begin{figure}
\begin{subfigure}[t]{0.4\linewidth}
\includegraphics[width=\linewidth]{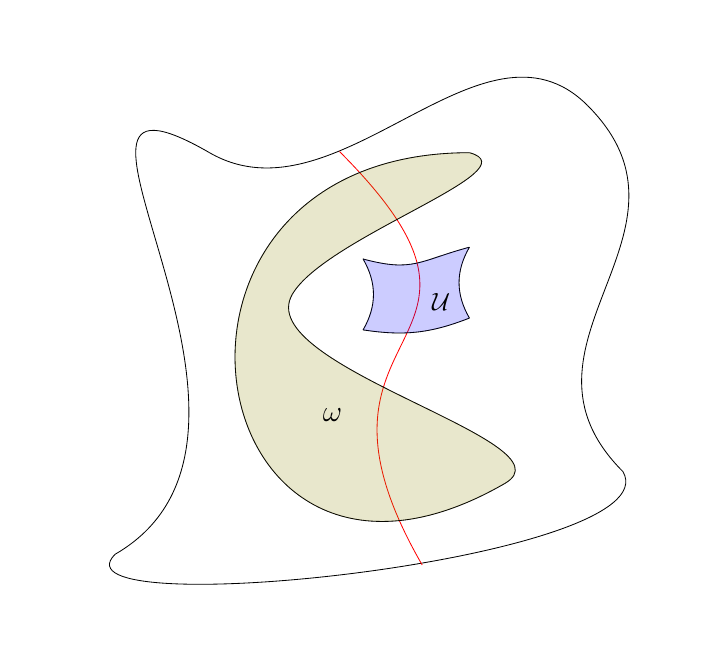}
\caption{An abstract manifold under \ref{H1i}.}
\end{subfigure}
\quad
\begin{subfigure}[t]{0.5\linewidth}
\includegraphics[width=\linewidth]{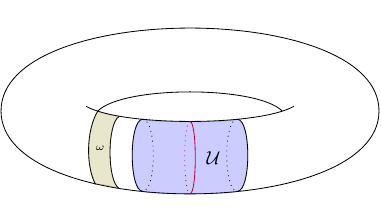}
\caption{The Torus under \ref{H1ii}.}
\end{subfigure}%
\caption{The singularity is in red, the control zone in green, and the tubular neighborhood in blue.}
\label{fig1}
\end{figure}

In such a tubular neighborhood, the sub-Riemannian structure is always generated by smooth vector fields of the form $X = \partial_x$ and $Y = \Tilde{q}(x,y)\partial_y$ (see \cite[Lemma 1, Theorem 1]{AgrachevGaussBonnet}).\\

We endow $\mathbb{M}$ with a smooth non-singular measure $\mu$, and we assume that there exists $L_0 \in (0,\min(d_{\operatorname{sR}}(\omega, \mathcal{Z}), \operatorname{inj}(\mathcal{Z})))$, such that in $\mathcal{U}$ diffeomorphic to $(-L_0,L_0) \times \Omega_y$, the vector fields $X, Y$, and the measure $\mu$, write as 
\begin{enumerate}[label=\textbf{H\arabic*},ref=\text{H\arabic*}, start = 2]
    \item \label{H2} $X = \partial_x$ and $Y = q(x)r(y)\partial_y$, with $q \in C^\infty([-L,L])$, $r \in C^\infty(\Omega_y)$, $r > 0$ and identically one whenever $\Omega_y = \mathbb{R}$, and $q$ satisfying
    \begin{align*}
        \partial_x^k q(0) = 0 \text{ for all } k \in \{0, ..., \gamma -1\}, \partial_x^\gamma q(0) > 0 \text{, and } q(x) \neq 0 \text{ for every } x \neq 0,
    \end{align*}
    \item \label{H3} $\mu = h(x)dxdy$, with $h > 0$. 
\end{enumerate}
Observe that given the form of $Y$ in \ref{H2}, the assumptions on $q$ and the fact that we need $\gamma$ iteration of Lie brackets only at the singularity $\mathcal{Z} \simeq \{0\} \times \Omega_y$ in \ref{H0} are equivalent.\\

We also emphasize that in particular, \ref{H2} says that the sub-Riemannian structure on $\mathcal{M}$ is generated, at least locally, by the smooth vector fields $\left\{\partial_x, q(x)r(y) \partial_y\right\}$. Moreover, it implies the assumption on the injectivity radius from $\mathcal{U} \cap \mathcal{Z}$ given in \ref{H0}.

\begin{remark}
    Under \ref{H1ii}, our results below will hold true if we write $\mu = h_1(x)h_2(y)dxdy$, with $h_2 > 0$. For the simplicity of the presentation, we make the choice of taking $\mu$ as prescribed by \ref{H3}. 
\end{remark}

We denote by $\Delta$ the sub-Laplacian with respect to $\mu$. That is, for every sufficiently regular function $f$, in $\mathcal{U}$ we have
\begin{align}\label{definition Delta}
    \Delta f = \operatorname{div}_\mu(\nabla f) = \frac{1}{h(x)}\partial_x(h(x) \partial_x f) + q(x)^2\partial_y(r(y)^2 \partial_y f).
\end{align}

Set $L^2(\mathcal{M}) := L^2(\mathcal{M}, \mu)$ the set of square integrable, with respect to $\mu$, real-valued functions on $\mathcal{M}$. We consider the Friedrich extension of $\Delta$ with minimal domain $C_c^\infty(\mathcal{M})$, and we keep this notation. That is, given a generating frame $\{X_1,...X_m\}$ for the sub-Riemannian structure, the form domain of $\Delta$, that is $D\left(\Delta^{1/2} \right)$, is the completion of $C_c^\infty(\mathcal{M})$ with respect to the horizontal norm 
\begin{align*}
    \|f\|_{D(\Delta^{1/2})}^2 := \int_\mathcal{M} \sum_{i=1}^m |X_i f|^2 \ d \mu, \quad f \in C_c^\infty(\mathcal{M}).
\end{align*}
The domain of $\Delta$ is $D(\Delta) = \left\{ f \in D\left(\Delta^{1/2} \right), \ \Delta f \in L^2(\mathcal{M}) \right\}$. Similarly, we denote by $\Delta_\mathcal{U}$ the restriction of $\Delta$ to $\mathcal{U}$, that is with domain $D(\Delta_\mathcal{U}) = \left\{ f \in D\left(\Delta_\mathcal{U}^{1/2} \right), \ \Delta_\mathcal{U} f \in L^2 \left(\mathcal{M} \right) \right\}$.

\subsection{The control problem and main result} 

The goal of this paper is to study the null-controllability properties of the following system. Let $T > 0$, 
\begin{align}\label{control grushin system M}
    \left\{ \begin{array}{lcll}
     \partial_t f - \Delta f & = & u(t,p)\mathbf{1}_\omega(p), & (t,p) \in (0,T) \times \mathcal{M},\\
     f(t,p) & = & 0, & (t,p) \in (0,T) \times \partial \mathcal{M}, \text{ if $\partial \mathcal{M} \neq \emptyset,$} \\
     f(0,p) & = & f_0(p), & p \in \mathcal{M},
    \end{array} \right.
\end{align}
where $f_0 \in L^2(\mathcal{M})$, $f \in L^2((0,T) \times \mathcal{M})$ is the state, and $u \in L^2((0,T)\times \mathcal{M})$ is the control supported in $\omega$. 

\begin{remark}
    From \ref{H0}, the operator $\Delta$ with domain $D(\Delta)$ is densely defined, self-adjoint on $L^2(\mathcal{M})$, and generates an analytic semigroup of contractions $(e^{t\Delta})_{t\geq0}$ on $L^2(\mathcal{M})$ (\cite[p. 261]{strichartz}). Moreover it is hypoelliptic (\cite[Theorem 1.1]{HormanderHypoell}). From \cite[Sec. 4.2]{Pazy2012-qf}), system \eqref{control grushin system M} is well posed in the sense that for every $f_0 \in L^2(\mathcal{M})$, there exists a unique solution $f \in C^0([0,T],L^2(\mathcal{M})) \cap L^2((0,T),D(\Delta^{1/2}))$ given by the Duhamel formula.
\end{remark}

\begin{definition}[Null-Controllability] \label{def nullcontr}
We say that system \eqref{control grushin system M} is null-controllable in time $T$ from $\omega$ if, for every $f_0 \in L^2(\mathcal{M})$, there exists $u \in L^2((0,T)\times \mathcal{M})$ such that the associated solution $f$ of system \eqref{control grushin system M} satisfies $f(T, \cdot, \cdot) = 0$ on $\mathcal{M}$.
\end{definition}

Let us introduce respectively the minimal time for null-controllability and the Agmon distance as 
\begin{align}
    T(\omega) &:= \inf \{ T > 0, \text{ such that the system is null-controllable in time $T$ from $\omega$} \},\\
    d_{\operatorname{agm}} &: x \in (-L,L) \mapsto \int_{\min(0,x)}^{\max(0,x)} |q(s)| \ ds. \label{agmon distance gamma = 1}
\end{align}

For a nonempty set $A \subset (-L,L)$, we may write in this paper $d_{\operatorname{agm}}(A) = \inf_{x \in A}d_{\operatorname{agm}}(x)$.\\

Observe that $T(\omega)=\infty$ if and only if the system of interest is not null-controllable independently of the final time $T>0$. 

\begin{theorem}\label{th nullcontr manifold} Assume \ref{H0} and \ref{H0'}, and that \ref{H2}-\ref{H3} holds in each of the settings of \ref{H1}. Depending on the value of $\gamma \geq 1$, we have the following statements for system \eqref{control grushin system M}. 
\begin{enumerate}[label=(\roman*), ref = \ref{th nullcontr manifold}(\roman*)]
    \item \label{th nullcontr manifold loc}Assume that $\gamma = 1$, that we are in the setting of \ref{H1i}, and $r$ is identically one. Then, there exists $L \in (0, \min(d_{\operatorname{sR}}(\omega, \mathcal{Z}), \operatorname{inj}(\mathcal{Z}))$ such that
    \begin{align}
        T(\omega) \geq  \frac{1}{q'(0)}\min \{ d_{\operatorname{agm}}(-L), d_{\operatorname{agm}}(L) \}.
    \end{align}
    \item \label{th nullcontr manifold glob} Assume that we are in the setting of \ref{H1ii}. If $\gamma > 1$, then $T(\omega) = \infty$, and if $\gamma = 1$, there exists $L \in (0, \min(d_{\operatorname{sR}}(\omega, \mathcal{Z}), \operatorname{inj}(\mathcal{Z}))$ such that
    \begin{align}
        T(\omega) \geq  \frac{1}{q'(0)}\min \{ d_{\operatorname{agm}}(-L), d_{\operatorname{agm}}(L) \}.
    \end{align} 
\end{enumerate}
\end{theorem}

\subsection{State of the art}

Around the seventies, V.V. Grushin, in \cite{Grušin_1971, Grušin_1970}, and M.S. Baouendi, in \cite{BSMF_1967__95__45_0}, introduced a class of degenerate hypoelliptic operators $-\partial_x^2 - x^{2\gamma}\partial_y^2$, with $\gamma > 0$, now commonly referred to as the classical (Baouendi-)Grushin operators. The first controllability result was given in \cite{beauchard2014null}, on the rectangle $(-1,1) \times (0,\pi)$ with Dirichlet boundary conditions. In this study, the control was supported on a vertical strip contained in $(0,1) \times (0,\pi)$, with its closure not intersecting the singularity $\{0\} \times (0,\pi)$. It is showed that when $\gamma < 1$, the equation is null-controllable in any time $T > 0$, that when $\gamma > 1$, null-controllability never occurs, and that when $\gamma = 1$, a minimal time is required to have null-controllability. This minimal time, for $\gamma = 1$, was latter obtained in \cite{beauchard20152d}, with the control supported on two vertical strips, in each side of the singularity, and for which the closure does not intersect the latter. The first result for the classical Grushin equation on $(-1,1) \times \mathbb{S}^1$, and on the rectangle, with the control supported on the complement of a horizontal strip, was given in \cite{CRMATH_2017__355_12_1215_0}. It is proved that when $\gamma = 1$, null-controllability is never achieved. In this similar setting for the control zone, but this time with the equation posed on $\mathbb{R}^2$, the same result was obtained in \cite{LISSYprolate2025}. In \cite{duprez2020control}, the minimal time and non null-controllability of the classical Grushin equation were thoroughly explored, considering a broader range of control support configurations.\\

The Grushin operator was then subsequently generalized by replacing $x^\gamma$ by a sufficiently regular function $q$ behaving like $x^\gamma$ as $x \rightarrow 0$. Similar results as for the classical operator were obtained. On the rectangle, the problem of boundary null-controllability has been investigated in \cite{beauchard2020minimal}. For the problem of internal null-controllability, for a wide range of considerations for the control zone, results were obtained in \cite{darde2023null}. On the Grushin sphere, endowed with the canonical measure inherited by $\mathbb{R}^3$, analogous results were obtained in \cite{tamekue2022null}. \\

Finally, considering the fractional Grushin equation, positive results were obtained very recently on $\mathbb{R}^n \times \mathbb{R}^m$ or $\mathbb{R}^n \times \mathbb{T}^m$ in \cite{jaming:hal-04245179} using spectral inequalities, and in \cite{letrouit2023observability} by means of resolvent estimates.\\

Although the question of controllability properties of the generalized Grushin operator was mainly investigated in some precise settings, the question of giving a geometric interpretation of these results on manifolds remains, as for a geometric interpretation of the null-controllability results of the sub-elliptic heat equation in general.\\

Concerning other results of null-controllability of a class of degenerate parabolic equation, we can also cite \cite{beauchard2017heat} that study the heat equation on the Heisenberg group, or results concerning the Kolmogorov equation \cite{LEROUSSEAU20163193, beauchardKolmog2015, koenigfractionalheat, beauchardkolmogorov} among others.

\subsection{Structure of the paper}
In Section \ref{section sketch of proofs}, we outline the proofs of our main result, and those of our complementary results.\\

In Section \ref{section spectral analysis}, we collect some results concerning the spectral analysis of the operators under consideration. In particular, in Section \ref{section dissipation speed classical grushin} and \ref{section exponential decay G_0}, we respectively remind of the behavior of the first eigenvalue, and exponential decay of the associated first eigenfunction, of the Fourier components of the classical Grushin operator
\begin{align}\label{eqn: def fourier components classical grushin}
    \mathsf{G}_{\xi} = -\partial_x^2 + \frac{q^{(\gamma)}(0)^2}{(\gamma!)^2}\xi^2x^{2\gamma}, \quad \xi \in \mathbb{R} \setminus \{0\}.
\end{align}

In Section \ref{section spectral G_V real}, we provide the asymptotic behavior of a sequence of eigenvalues of the Fourier components of the generalized Grushin operator 
\begin{align}\label{eqn: def fourier components generalized grushin}
    G_{V,\xi} = - \partial_x^2 + \xi^2q(x)^2 + V(x), \quad \xi \in \mathbb{R} \setminus \{0\}.
\end{align}
The study of the exponential decays of the associated eigenfunctions are encapsulated within the proofs of the theorems. In Section \ref{section spectral G_V complex}, we study the behavior of the first eigenvalue of the operator $G_{V,\xi}$ for $\xi \in \mathbb{C}$. \\

In Section \ref{section grushin euclidiens}, we study the non null-controllability of initial states in form domains, for the generalized Grushin equation posed on Euclidean domains. When the control acts on vertical strips at non-negative distance from the singularity, for $\gamma \geq 1$, we prove a non null-controllability result  in Section \ref{section proof th euclid bounded}. For $\gamma = 1$, and the control acting on the complement of a rectangle, we provide a non null-controllability result in Section \ref{section proof th euclid bounded koenig}. \\

Finally, in Section \ref{section proofs of theorem on manifolds}, we prove Theorem \ref{th nullcontr manifold}.

\section{Sketch of proofs}\label{section sketch of proofs}

Let us begin by outlining the proof of Theorem \ref{th nullcontr manifold}, which is presented in Section \ref{section proofs of theorem on manifolds}. We fix $\mathcal{U}$ a tubular neighborhood prescribed by \ref{H1i} or \ref{H1ii}, and satisfying additional conditions that will be precised throughout the proof. Theorem \ref{th nullcontr manifold} will be proved by contradiction, together with cutoff arguments. \\

The underlying idea is the following. We assume that system \eqref{control grushin system M} is null-controllable in time $T>0$ from $\omega$ satisfying one of the conditions of \ref{H1}. We fix $\mathcal{U}$ as prescribed by the sub-Riemannian setting, and sufficiently small such that at least \ref{H2} and \ref{H3} hold (and additional conditions made precise in Section \ref{section proofs of theorem on manifolds}). \\ 

Given an initial state $f_0 \in L^2(\mathcal{U})$, we can construct under the null-controllability assumption of system \eqref{control grushin system M}, by means of cutoffs arguments, a control $u$ supported in an arbitrary small neighborhood of $\partial \mathcal{U} \setminus \partial \mathbb{M}$, that we denote by $\tilde{\omega}$, such that the associated solution of 
\begin{align}\label{eqn: system U sketch}
        \left\{ \begin{array}{lcll}
        \partial_t f - \Delta_\mathcal{U} f & = & \mathbf{1}_{\tilde{\omega}}(x,y) u(t,x,y), & (t,x,y) \in (0,T) \times \mathcal{U},\\
        f(t,x,y) & = & 0, & (t,x,y) \in (0,T) \times \partial \mathcal{U}, \\
        f(0,x,y) & = & f_0(x,y), & (x,y) \in \mathcal{U},
        \end{array} \right.
    \end{align}
satisfies $f(T) = 0$. We call this a \textit{reduction process}. \\

However, when working under \ref{H1i}, in system \eqref{eqn: system U sketch} the control $u$ constructed as above, and steering $f_0 \in L^2(\mathcal{U})$ to $0$, is not necessarily in $L^2((0,T) \times \tilde{\omega})$. This is due to the loss of regularity near the singularity $\mathcal{Z}$. A sufficient condition ensuring that the induced control belongs to $L^2((0,T) \times \tilde{\omega})$ is that the initial state in $\mathcal{U}$ is taken sufficiently regular.\\

Thus, for the control system \eqref{eqn: system U sketch}, with state space $L^2(\mathcal{U})$, we study the null-controllability only for sufficiently regular initial states, which in our case, under \ref{H1i}, are to be taken in $D \left(\Delta_\mathcal{U}^{1/2} \right)$ (see Proposition \ref{proposition internal M to U}). To avoid any ambiguity and to ensure uniformity across the statements and proofs, we will therefore study the null-controllability of initial states in the form domain for any of our control systems in $\mathcal{U}$, regardless of whether we work under \ref{H1i} or \ref{H1ii}. \\ 

To study the null-controllability of initial states in the form domain, of system \eqref{eqn: system U sketch}, we address the latter within its coordinates representation, where $\Delta_\mathcal{U}$ is expressed as \eqref{definition Delta}. After the transformation $f \mapsto \sqrt{h} f$, we are left with addressing the null-controllability properties of initial states taken in the form domain of the spatial operator properly introduced below in \eqref{def: GV}, of the following system on $L^2(\Omega,dxdy)$, where $\Omega = \Omega_x \times \Omega_y$ is given in the sub-Riemannian setting, the control zone satisfies one of the conditions outlined in \ref{H1}, and $V \in L^\infty(\Omega_x)$ depends on the measure $h(x)dxdy$. Let $T > 0$, 
\begin{align} \label{control system change variable}
\left\{ \begin{array}{lcll}
     \partial_t f - \partial_x^2f - q(x)^2\partial_y(r(y)^2 \partial_y f) + V(x)f & = & u(t,x,y)\mathbf{1}_{\Tilde{\omega}}(x,y), & (t,x,y) \in (0,T) \times \Omega,\\
     f(t,x,y) & = & 0, & (t,x,y) \in (0,T) \times \partial \Omega, \\
     f(0,x,y) & = & f_0(x,y), & (x,y) \in \Omega.
    \end{array} \right.
\end{align}
Consequently, Theorem \ref{th nullcontr manifold loc} and \ref{th nullcontr manifold glob} follow, respectively, from Theorem \ref{th gamma = 1 euclidean koenig} and Theorem \ref{th gamma >= 1 euclidean bounded} below. \\

The \textit{reduction process} is treated in Section \ref{section reduction process}, while Theorem \ref{th nullcontr manifold loc} and \ref{th nullcontr manifold glob}  are both proved in Section \ref{section proof of Theorem nullcontr manifold}.\\

To allow us to work exclusively in the Euclidean setting, let us denote by $G_V$ the operator
\begin{align*}
    G_V = - \partial_x^2 - q(x)^2\partial_y(r(y)^2 \partial_y) + V(x),
\end{align*}
with minimal domain $C_c^\infty(\Omega)$. We consider its Friedrich extension, still denoted by $G_V$, and defined by 
\begin{align}\label{def: GV}
    \begin{array}{ccl}
        D(G_V) &=& \left\{ f \in D\left(G_V^{1/2}\right), \ G_Vf \in L^2(\Omega) \right\}  \\[8pt]
         G_V &=& - \partial_x^2f - q(x)^2\partial_y(r(y)^2 \partial_y) + V(x), 
    \end{array}
\end{align}
where $D\left(G_V^{1/2} \right)$ is the completion of $C_c^\infty(\Omega)$ with respect to the norm 
\begin{align*}
    \|f\|_{D\left(G_V^{1/2}\right)}^2 = \int_\Omega |\partial_xf|^2 + q(x)^2r(y)^2|\partial_y f|^2 + V(x) |f|^2 \ dx \ dy, \quad f \in C_c^\infty(\Omega).
\end{align*}

\begin{theorem}[Control in the complement of a rectangle]\label{th gamma = 1 euclidean koenig}
Let $\gamma = 1$. Consider system \eqref{control system change variable} with $\Omega_x = (-L,L)$, $\Omega_y = (0, \pi)$, and $\Tilde{\omega}$ the complement of a rectangle $[-a,b] \times \overline{I}$, $a,b > 0$, where $\overline{I}$ is a proper closed interval of $\Omega_y$. Assume that $q$ satisfies \ref{H2}, $r$ is identically one on $\Omega_y$, and $V \in L^\infty(\Omega)$. Then, for any time $T > 0$ such that
\begin{align}
     T < \frac{1}{q'(0)}\min (d_{\operatorname{agm}}(-a),d_{\operatorname{agm}}(b)),
\end{align}
where $d_{\operatorname{agm}}(x)$ is introduced in \eqref{agmon distance gamma = 1}, there exists $f_0 \in D\left(G_V^{1/2}\right)$ that cannot be steered to $0$ in $L^2(\Omega)$ by means of an $L^2((0,T) \times \tilde{\omega})$-control.
\end{theorem}

\begin{theorem}[Control on vertical strips]  \label{th gamma >= 1 euclidean bounded}
Consider system \eqref{control system change variable} with $\Omega_x = (-L,L)$, $\Omega_y =  (0,\pi) \text{, } \mathbb{S}^1$ or $\mathbb{R}$. Assume that $\Tilde{\omega} = \Tilde{\omega}_x \times \Omega_y$, with $\operatorname{dist}(\Tilde{\omega}_x,0) > 0$, q satisfies \ref{H2}, and $V \in L^\infty(\Omega)$. 
\begin{enumerate}[label=(\roman*)]
\itemsep0.5em
    \item If $\gamma = 1 $, then, for any time
    \begin{align}
        T < \frac{d_{\operatorname{agm}}(\tilde{\omega}_x)}{q'(0)}
    \end{align}
    there exists $f_0 \in D\left(G_V^{1/2}\right)$ that cannot be steered to $0$ in $L^2(\Omega)$ by means of an $L^2((0,T) \times \tilde{\omega})$-control.
    \item If $\gamma >1 $, then, for any time $T > 0$, there exists $f_0 \in D\left(G_V^{1/2}\right)$ that cannot be steered to $0$ in $L^2(\Omega)$ by means of an $L^2((0,T) \times \tilde{\omega})$-control.
\end{enumerate}
\end{theorem}

Recall that when $\Omega_y = \mathbb{R}$, the function $r$ is assumed to be identically one.\\

In particular in the above two theorems we obtain some lower bounds on $T(\tilde{\omega})$ for system \eqref{control system change variable}. But these theorems are a little bit more precise as they state that even regular initial states may not be steerable to $0$.

\begin{remark}
    Theorems \ref{th gamma = 1 euclidean koenig} and \ref{th gamma >= 1 euclidean bounded} stay true for initial states in $D(G^s)$, for any $s \geq 0$, with $D(G^0) = L^2(\Omega)$. This will be clear from the proof for Theorem \ref{th gamma >= 1 euclidean bounded}. For this statement in the setting of Theorem \ref{th gamma = 1 euclidean koenig}, we refer the reader to Remark \ref{rmk: disprove control regular solutions}.
\end{remark}

\begin{remark}
    Smoothness of the functions $q$ and $r$ is not needed in Theorems \ref{th gamma = 1 euclidean koenig} and \ref{th gamma >= 1 euclidean bounded}. One can simply require $q \in C^{\gamma+1}([-L,L])$ and $r \in C^1(\Omega_y)$.
\end{remark}

\begin{remark}
    In the case $V = 0$ and for initial states in $L^2(\Omega)$, Theorem \ref{th gamma = 1 euclidean koenig} is already known from \cite[Theorem 1.3]{darde2023null}. 
\end{remark}

\begin{remark}\label{remark sketch of proofs}
While Theorems \ref{th gamma = 1 euclidean koenig} and \ref{th gamma >= 1 euclidean bounded} assume that $\Omega_x$ is symmetric with respect to zero, they remain valid when substituting $(-L,L)$ with any other interval containing $0$ in its interior (see Section \ref{section comments}). Same remark goes for $\Omega_y$, that can be replaced by any bounded interval instead of $(0,\pi)$.
\end{remark}

As customary, we will adopt the observability perspective. By linearity of our systems, the null-controllability in time $T>0$ of initial states in $D\left(G_V^{1/2}\right)$ of system \eqref{control system change variable} is equivalent to a $D\left(G_V^{-1/2}\right)-L^2$ observability inequality for its adjoint system (see \textit{e.g.} \cite[Theorem 2.44]{coron2007control}), that we simply call an observability inequality. 

\begin{definition}[$D\left(G_V^{-1/2}\right)-L^2$ observability]
Let $T>0$. We say that the adjoint system
\begin{align} \label{adjoint system change variable}
\left\{ \begin{array}{lcll}
     \partial_t f - \partial_x^2f - q(x)^2\partial_y(r(y)^2 \partial_y f) + V(x)f & = & 0, & (t,x,y) \in (0,T) \times \Omega,\\
     f(t,x,y) & = & 0, & (t,x,y) \in (0,T) \times \partial \Omega, \\
     f(0,x,y) & = & f_0(x,y), & (x,y) \in \Omega,
    \end{array} \right.
\end{align}
is observable from $\Tilde{\omega}$ in time $T$ if there exists $C>0$ such that for every $f_0 \in L^2(\Omega)$, the associated solution satisfies
\begin{align}\label{observability inequality Omega}
\|f(T)\|_{D\left(G_V^{-1/2}\right)}^2 \leq C \int_0^T \int_{\Tilde{\omega}} |f(t,x,y)|^2 \ dx \ dy \ dt.
\end{align}
\end{definition}

Let us start with Theorem \ref{th gamma >= 1 euclidean bounded} and with $\Omega_y$ bounded. Denote by $(\phi_n)_n$ the eigenfunctions of $-\partial_y(r(y)^2\partial_y)$, associated to $\xi_n^2 \rightarrow + \infty$, which form an Hilbert basis of $L^2(\Omega_y)$. We follow the classical idea, first proposed in \cite{beauchard2014null}, of testing the inequality \eqref{observability inequality Omega} against the sequence of solutions
\begin{align}\label{eqn: form sequence solutions}
    g_n(t,x,y) = e^{-\lambda_nt}v_n(x)\phi_n(y),
\end{align}
where $v_n$ is a normalized in norm eigenfunction, associated to $\lambda_n$, of 
\begin{align}\label{eqn: fourier component generalized operator}
    G_{V,\xi_n} = -\partial_x^2 + \xi_n^2q(x)^2 + V(x).
\end{align}

Assuming that the observability inequality \eqref{observability inequality Omega} holds from $\tilde{\omega}_x \times \Omega_y$ in time $T > 0$, and testing it against the solutions $g_n$ of the form \eqref{eqn: form sequence solutions}, thus implies the existence of a constant $C > 0$ such that, for every $n \geq 1$, 
\begin{align}\label{eqn: uniform observability eigenfunctions}
    \|g_n(T)\|_{D\left(G_V^{-1/2}\right)}^2 = \frac{e^{-2\lambda_nT}}{\lambda_n}\leq C T \int_{\tilde{\omega}_x} |v_n(x)|^2 \ dx \ dt.
\end{align}

We therefore have to compare the dissipation on left of \eqref{eqn: uniform observability eigenfunctions}, with the decay of the eigenfunctions $v_n$ on the set $\tilde{\omega}_x$ in the limit $n \rightarrow + \infty$. This amounts to a precise spectral analysis of the Fourier components \eqref{eqn: fourier component generalized operator} of our generalized Grushin operator, and a precise analysis of the behavior of its eigenfunctions. \\

To extract the eigenvalue $\lambda_n$ and have good estimates on it, the idea is to see the generalized operator $G_{V,\xi_n}$  \eqref{eqn: fourier component generalized operator} as a perturbation of the classical operator in Fourier $\mathsf{G}_{\xi_n} = -\partial_x^2 + \frac{q^{(\gamma)}(0)^2}{(\gamma!)^2}\xi_n^2x^{2\gamma}$. Since the classical eigenfunctions concentrate near $x = 0$, where the generalized operator behaves like the classical one by assumption \ref{H2}, we expect $\lambda_n$ to behave asymptotically the same way as the sequence of first eigenvalues of $\mathsf{G}_{\xi_n}$. The proof culminates with the application of an Agmon-type arguments, showing that the eigenfunctions $v_n$ concentrate outside the control zone significantly faster than they decay across the entire domain. This behavior occurs due to the degeneracy of $q$.\\

The spectral analysis will be performed replacing $\xi_n$ by a generic parameter $\xi$. For the Fourier components $\mathsf{G}_{\xi}$ \eqref{eqn: def fourier components classical grushin} of the classical operator, the analysis is provided in Section \ref{section reminders G_0}. For the Fourier components $G_{V,\xi}$ \eqref{eqn: def fourier components generalized grushin} of the generalized operator, see Section \ref{section spectral G_V real}. The proof of Theorem \ref{th gamma >= 1 euclidean bounded} is provided along with the Agmon-type argument in Section \ref{section proof th euclid bounded}. \\

When $\Omega_y = \mathbb{R}$, we treat the problem exclusively in Fourier. We choose as solutions of the Fourier transform of system \eqref{adjoint system change variable} the sequence $g_n(t,x,\xi) = v(x,\xi)e^{-\lambda(\xi) t}\psi_n(\xi)$, where $v(\cdot,\xi)$ is an eigenfunction of $G_{V,\xi}$ \eqref{eqn: def fourier components generalized grushin}, associated to $\lambda(\xi)$, and $\psi_n$ is a sequence of cutoffs that localizes in high frequencies. The strategy is then the same as for the case $\Omega_y$ bounded. \\

For Theorem \ref{th gamma = 1 euclidean koenig}, we follow \cite{darde2023null}. Much of the work is already carried out in \cite{darde2023null}, so we simply show that their strategy still holds in our case. Let us present the idea when $\Omega_y = \mathbb{S}^1$. By the Fourier decomposition with respect to $y$ presented above, we can choose a sequence of solutions of the form $g_N(t,x,y) = \sum_{n \geq N} a_n v_n(x)e^{iny}e^{-\lambda_n t}$, where the sum is finite. Here, $v_n$ is chosen to lie within the first eigenspace of the operator $-\partial_x^2 + n^2 q(x)^2 + V(x)$, associated to the first eigenvalue $\lambda_n$ of the latter. Since near $x = 0$ we have that $-\partial_x^2 + n^2 q(x)^2 + V(x) \approx -\partial_x^2 + n^2q'(0)^2x^2 + V(x)$, in the limit $n \rightarrow + \infty$ classical arguments from complex perturbation theory will show that $\lambda_n \sim nq'(0)$, as the effect of the potential $V$ becomes more and more negligible in front of $n^2q^2$. Choosing $v_n$ as the spectral projection of $\Tilde{v}_n(x) = n^{1/4}e^{-nq'(0)x^2/2}$, we show using an Agmon-type argument that it concentrates near zero as $v_n \approx e^{-nd_{\operatorname{agm}}(x)}$ (which can also been show by a standard WKB argument). Hence, for $N$ large enough, our solutions behave like $g_N(t,x,y) \approx \sum_{n \geq N} a_n e^{-nd_{\operatorname{agm}}(x)}e^{iny}e^{-nq'(0)t}$, or more precisely 
\begin{align*}
    g_N(t,x,y) = \sum_{n \geq N} a_{n} \gamma_{t,x}(n) e^{n(iy - q'(0)t - (1-\epsilon)d_{\operatorname{agm}}(x))},
\end{align*}
where the $\gamma_{t,x}(n)$ are error terms.\\

We observe then that after a change of variables, our solutions resemble complex polynomials of the form $\sum_{n \geq N} a_{n}z^n$, up to the error terms $\gamma_{t,x}(n)$. The idea is then to show that observability implies an $L^2-L^\infty$ inequality on these polynomials with a zero of order $N$ at zero, which cannot hold in small times. This implication is carried out up to the possibility of estimating the error term, and is based on rather complicated complex analysis arguments that we shall not exhibit here (see \cite{duprez2020control,darde2023null}). We refer the reader to Section \ref{section proof th euclid bounded koenig} (and \cite{duprez2020control,darde2023null}) for more details. The complex spectral analysis, and the proof of Theorem \ref{th gamma = 1 euclidean koenig}, are provided respectively in Sections \ref{section spectral G_V complex} and \ref{section proof th euclid bounded koenig}.

\section{Spectral analysis}\label{section spectral analysis}

We want to extract a sequence of eigenvalues of the operator $G_V = - \partial_x^2 - q(x)^2\partial_y(r(y)^2 \partial_y ) + V(x) $ on $L^2(\Omega)$, with $\Omega_y = (0,\pi)$ or $\mathbb{S}^1$, for which we can get sufficiently precise estimates, and for which we can estimate the decay of some associated eigenfunctions in the control zone. We are therefore interested by the eigenvalue problem
\begin{align}\label{eqn: eigenvalue problem grushin generalized}
    \left\{
    \begin{array}{llll}
        - \partial_x^2 f - q(x)^2\partial_y(r(y)^2 \partial_y f) + V(x)f &=& \lambda f, &(x,y) \in \Omega,  \\
         f(x,y) &=& 0, &(x,y) \in \partial\Omega.
    \end{array}\right.
\end{align}

As explained in Section \ref{section sketch of proofs}, this is done by studying the Fourier components of $G_V$, that we denoted by $G_{V,\xi_n}$. In the analysis of this section, we replace $\xi_n$ by a generic parameter $\xi$. 

\subsection{Preliminaries on the classical Grushin operator}\label{section reminders G_0}

We introduce, for every $\xi \in \mathbb{R} \setminus \{0\}$, and for every $\gamma \geq 1$, the operator $\mathsf{G}_{\xi}$ on $L^2(\Omega_x)$, defined by 
\begin{align}\begin{array}{ccl}
D(\mathsf{G}_{\xi}) &:=& H^2 \cap H_0^1(\Omega_x), \\
    \mathsf{G}_{\xi} u &:=& -u'' + \frac{q^{(\gamma)}(0)^2}{(\gamma!)^2}\xi^2|x|^{2\gamma}u.
\end{array}
\end{align}
$\mathsf{G}_{\xi}$ has compact resolvent, is self-adjoint, positive definite, and has discrete spectrum \cite{Berezin1991-nn}. Hence, its first eigenvalue is given by the Rayleigh formula
\begin{align}\label{mu_xi}
\mu_\xi =  \min \{ \mathcal{Q}_{\xi,L} (u,u) \text{; $u \in D(\mathsf{G}_{\xi}), \|u\|_{L^2(-L,L)}=1 $} \},
\end{align}
with 
\begin{align}
\mathcal{Q}_{\xi,L}(u,v)  = \displaystyle\int_{-L}^L u'(x)v'(x) + \frac{q^{(\gamma)}(0)^2}{(\gamma!)^2}\xi^2|x|^{2\gamma}u(x)v(x) \ dx.
\end{align}

\subsubsection{Dissipation speed}\label{section dissipation speed classical grushin}

The following Proposition is already known \cite[Proposition 4]{beauchard2014null}, but we propose a slightly different proof with a more precise upper bound. This gain of precision will not be of use for our proofs, but we provide it for its own interest. First, we need the following two Lemmas.

\begin{lemma}{\cite[Lemma 3.5.]{prandi2018quantum}}\label{energy cutoff lemma}
    Let $\chi$ be a real-valued Lipschitz function with compact support in $\mathbb{R}$. Let $u\in H^1(\mathbb{R})$. Then, we have 
    \begin{align}
        \mathcal{Q}_{1,\infty} (\chi u,\chi u) = \mathcal{Q}_{1,\infty} (u,\chi^2 u) + \langle u , |\chi'|^2 u \rangle_{L^2(\mathbb{R})}.
    \end{align}
\end{lemma}

\begin{lemma}\label{lemma: construction bump functions}
    There exists $C > 0$, such that for every $R > 0$ and $\delta \in (0,1)$, there exist $\chi_1,\chi_2 \in C^\infty(\mathbb{R})$ two smooth functions such that
\begin{itemize}
    \item $0 \leq \chi_i(x) \leq 1$, for every $x \in \mathbb{R}$, and $i=1,2$, 
    \item $\chi_1 = 1$ on $[-\delta R,\delta R]$ and $\chi_1 = 0$ on $\mathbb{R}\setminus[-R,R]$,
    \item $\chi_2 = 1$ on $\mathbb{R}\setminus[-R,R]$ and $\chi_2 = 0$ on $[-\delta R,\delta R]$,
    \item $\chi_1(x)^2 + \chi_2(x)^2 = 1$, for every $x \in \mathbb{R}$, 
    \item $\underset{x \in \mathbb{R}}{\sup} |\chi'_i(x)| \leq  C \displaystyle \frac{\pi}{2(1-\delta)R}$, for $i=1,2$. 
\end{itemize}
\end{lemma}

\begin{proof}
    Without loss of generalities, we present the construction for $x> 0$. The case $x < 0$ follows by symmetry. We first set $f \in C^\infty(\mathbb{R})$ to be defined by 
    \begin{align} f(x) = \left\{
        \begin{array}{ll}
            e^{-1/x}, & x > 0,   \\
             0, & x \leq 0. 
        \end{array}\right.
    \end{align}
    Next, we define $g \in C^\infty(\mathbb{R})$ by
    \begin{align}
        g(x) = \frac{f(x)}{f(x) + f(1-x)},
    \end{align}
    which satisfies 
    \begin{align*}
        \begin{array}{llll}
            g(x) &=& 0, & \text{for every } x \leq 0,  \\[6pt]
            g(x) &=& 1, & \text{for every } x \geq 1,\\[6pt]
            g'(x) &=& 0, & \text{for every } x \leq 0 \text{ and } x \geq 1,\\[6pt]
            |g'(x)| &\leq& C & \text{for every } x \in \mathbb{R}.
        \end{array}
    \end{align*}
    Finally, we set
    \begin{align}
        s(x) = \frac{\pi}{2} g \left( \frac{x-\delta R}{(1-\delta) R} \right),
    \end{align}
    
    which satisfies $s \in C^\infty(\mathbb{R})$ and
    \begin{align*}
        \begin{array}{llll}
            s(x) &=& 0, & \text{for every } x \leq \delta R,  \\[6pt]
            s(x) &=& \pi/2, & \text{for every } x \geq R,\\[6pt]
            s'(x) &=& 0, & \text{for every } x \leq \delta R \text{ and } x \geq R,\\[6pt]
            |s'(x)| &\leq& C \displaystyle \frac{\pi}{2(1-\delta)R} & \text{for every } x \in \mathbb{R}.
        \end{array}
    \end{align*}
    We define $\chi_1$ and $\chi_2$ on $(0,+\infty)$ by 
    \begin{align}
        \chi_{\xi,1}(x) &= \cos(s(x)), \\
        \chi_{\xi,2}(x)& = \sin(s(x)).
    \end{align}
    One easily verifies that $\chi_1$ and $\chi_2$ satisfy the sought properties for $x > 0$.
\end{proof}

We can now estimate $\mu_\xi$.

\begin{proposition}\label{proposition dissipation classical Omega bounded}
There exists $ C > 0$ such that 
    \begin{align}\label{lower bound mu_n}
         \mu_\xi \geq C|\xi|^{\frac{2}{1+\gamma}}, \quad \text{for every $|\xi| > 0$,}
    \end{align}
and such that for every $\epsilon > 0$ sufficiently small, there exists $R > 0$, such that for every $|\xi| \geq R$, we have
\begin{align}\label{upper bound mu_n}
    \mu_\xi \leq (1 + \epsilon)( C + \epsilon)|\xi|^{\frac{2}{1+\gamma}}.
\end{align}
\end{proposition}

\begin{proof} We start with the lower bound.
Let $\tau_\xi := |\xi|^{1/(1 + \gamma)}$. Making the change of variable $s = \tau_\xi x$, and setting $v(s) = u(\tau_{\xi}^{-1}y)\tau_{\xi}^{-1/2}$, \eqref{mu_xi} becomes
\begin{align}
    \mu_\xi =  \tau_{\xi}^2 \min \{ \mathcal{Q}_{1,\tau_\xi L}(v,v), v \in H^2 \cap H_0^1(-\tau_\xi L, \tau_\xi L), \|v\|_{L^2(-\tau_\xi L, \tau_\xi L)} = 1\}.
\end{align}

Thus, we have
 \begin{align}
     \mu_\xi \geq \tau_{\xi}^2 \cdot \overline{\mu},
 \end{align}
with $\overline{\mu} := \inf \{ \mathcal{Q}_{1,\infty}(v,v), v \in H^2(\mathbb{R}), |x|^\gamma v \in L^2(\mathbb{R}), \|v\|_{L^2(\mathbb{R})} = 1\}> 0$.\\

We now treat the upper bound. Let $0 < \delta < 1$. Set $\chi_{\xi,1}$, $\chi_{\xi,2}$ to be the two smooth functions prescribed by Lemma \ref{lemma: construction bump functions} with $R = L\tau_\xi$. Thus, $\chi_{\xi,1}$ and $\chi_{\xi,2}$ satisfy
\begin{itemize}
    \setlength\itemsep{1em}
    \item $0 \leq \chi_{\xi,i}(x) \leq 1$, for $i=1,2$, 
    \item $\chi_{\xi,1}(x) = 1$ on $[-\delta L\tau_\xi,\delta L\tau_\xi]$ and $\chi_{\xi,1}(x) = 0$ on $\mathbb{R}\setminus[-L\tau_\xi,L\tau_\xi]$,
    \item $\chi_{\xi,2}(x) = 1$ on $\mathbb{R}\setminus[-L\tau_\xi,L\tau_\xi]$ and $\chi_{\xi,2}(x) = 0$ on $[-\delta L\tau_\xi,\delta L\tau_\xi]$,
    \item $\chi_{\xi,1}(x)^2 + \chi_{\xi,2}(x)^2 = 1$, for every $x \in \mathbb{R}$ and $|\xi| > 0$, 
    \item $\underset{x \in \mathbb{R}}{\sup} |\chi'_{\xi,i}(x)| \leq  C \displaystyle \frac{\pi}{2(1-\delta)L\tau_\xi}$, for $i=1,2$ and for some $C >0$ independent of $\xi$ and $\delta$. \\
\end{itemize}
Let $v \in H^1(\mathbb{R})\cap L^2(|s|^{2\gamma}ds)$, $\|v\|_{L^2(\mathbb{R})}=1 $, be the minimizer of 
$\mathcal{Q}_{1,\infty}(v,v)$. From Lemma \ref{energy cutoff lemma}, we derive that  
\begin{align}\label{eqn: quadratic form bound below}
    \mathcal{Q}_{1,\infty}(v,v) = \sum_{i=1}^2 \mathcal{Q}_{1,\infty}(\chi_{\xi,i}v,\chi_{\xi,i}v) - \sum_{i=1}^2 \int_\mathbb{R} |\chi'_{\xi,i}|^2|v|^2 \geq \mathcal{Q}_{1,\infty} (\chi_{\xi,1}v,\chi_{\xi,1}v) - c(\xi)\|v\|_{L^2(\mathbb{R})},
\end{align}
where $c(\xi) := 2\left( C \displaystyle \frac{\pi}{2(1-\delta)L\tau_\xi} \right)^2$. \\

Henceforth, by definition of $\mu_\xi$, with the minimum taken on all non-zero elements, we have that 
\begin{align*}
    \mu_\xi &\leq \tau_{\xi}^2 \frac{\mathcal{Q}_{1,\tau_\xi L}(\chi_{\xi,1}v,\chi_{\xi,1}v)}{\|\chi_{\xi,1}v\|_{L^2(-\tau_\xi L,\tau_\xi L)}^2} \\
    &= \tau_{\xi}^2 \frac{\mathcal{Q}_{1,\infty} (\chi_{\xi,1}v,\chi_{\xi,1}v)}{\|\chi_{\xi,1}v\|_{L^2(\mathbb{R})}^2} \\
    &\leq \frac{\tau_{\xi}^2}{\|\chi_{\xi,1}v\|_{L^2(\mathbb{R})}^2} \left( \mathcal{Q}_{1,\infty} (v,v) + c(\xi) \right), \quad \text{by \eqref{eqn: quadratic form bound below}, and because $\|v\|_{L^2(\mathbb{R})} = 1$}, \\
    &=\frac{\tau_{\xi}^2}{\|\chi_{\xi,1}v\|_{L^2(\mathbb{R})}^2} \left( \overline{\mu} + c(\xi) \right).
\end{align*}
This proves the upper bound since $\|\chi_{\xi,1}v\|_{L^2(\mathbb{R})} \rightarrow 1$ and $c(\xi) \rightarrow 0$ as $|\xi|$ tends to infinity.
\end{proof}

In the case $\gamma = 1$, we can actually get very much more precise estimates as we know exactly the first eigenfunction of the harmonic oscillator on $\mathbb{R}$. This is the idea of \cite[Lemma 4]{beauchard2014null}.

\begin{proposition}\label{prop: asymptotic first eigenvalue classical operator}
    Let $\gamma = 1$. Then, 
    \begin{align}
        \mu_\xi \sim q'(0)|\xi| \text{, as $|\xi| \to +\infty$.}
    \end{align}
\end{proposition}

\subsubsection{Exponential decay of the eigenfunctions}\label{section exponential decay G_0}

In this section, we recall the exponential decay for the first eigenfunctions of $\mathsf{G}_{\xi}$, which follows from the above estimates on the first eigenvalues. The propositions in this section are already proved in \cite{beauchard2014null} for $n \in \mathbb{N}^*$, $\Omega_x = (-1,1)$, and the operator $-\partial_x^2 + n^2\pi^2 x^{2\gamma}$, but hold in our case. So we state them, but omit their proofs. We first have the following Lemma, which is proved as \cite[Lemma 2]{beauchard2014null}.

\begin{lemma}
There exists a unique non-negative function $v_\xi\in L^2(\Omega_x)$ such that $\|v_\xi\|_{L^2(\Omega_x)}=1$ and which solves the problem
\begin{align}\label{system v_xi}
\left\{
    \begin{array}{llll}
-\partial_x^2 v_\xi + \frac{q^{(\gamma)}(0)^2}{(\gamma!)^2}\xi^2|x|^{2\gamma}v_\xi &=& \mu_\xi v_\xi,& \text{ if } x \in \Omega_x, \\
v_\xi(\pm L) &=& 0,& \text{ if $\Omega_x = (-L,L)$,}\\
\lim_{|x| \rightarrow +\infty} v_\xi(x) &=& 0, & \text{ if }\Omega_x = \mathbb{R},
    \end{array}\right.
\end{align}
Moreover, $v_\xi$ is even.
\end{lemma}

We now have the following asymptotic $L^2(\Omega_x)$ exponential decay of the $v_\xi$'s, as $|\xi| \rightarrow + \infty$, outside any neighborhood of $0$ in $(-L,L)$. It is proved in \cite[Lemma 3]{beauchard2014null}.

\begin{proposition}\label{proposition supersolution grushin model}
For every $|\xi| > 0$ large enough, set
\begin{align}\label{defxxi}
   x_\xi = \left( \frac{(\gamma!)^2\mu_\xi}{q^{(\gamma)}(0)^2\xi^2}\right)^{1/2\gamma}.
\end{align}

Then, there exists $R > 0$, such that for every $ |\xi| > R$, there exists a function $W_\xi$ of the form 
\begin{align}
    W_\xi(x) = A_\xi e^{-B_\xi x^{\gamma + 1}},
\end{align}
with, for some $C > 0$, 
\begin{align}
A_\xi = 2\frac{\sqrt{x_\xi} \mu_\xi}{(\gamma+1)B_\xi x_\xi^\gamma} e^{B_\xi x_\xi^{\gamma + 1}} \text{ and } B_\xi \sim C|\xi| \text{, as $|\xi| \rightarrow +\infty$,} 
\end{align}
such that for every $x \geq x_\xi$,
\begin{gather}
v_\xi(x) \leq W_\xi(x).
\end{gather}
Moreover, $W_\xi \in C^2(\{x \geq x_\xi\},\mathbb{R})$ is solution of
\begin{align}\label{systemWxi}
\left\{
    \begin{array}{llll}
-\partial_x^2 W_\xi + \left( \frac{q^{(\gamma)}(0)^2}{(\gamma!)^2}\xi^2|x|^{2\gamma} - \mu_\xi \right) W_\xi &\geq& 0,& \text{ if } x > x_\xi, \\
W_\xi(L) &\geq& 0,& \\
\partial_x W_\xi(x_\xi) &<& -\sqrt{x_\xi}\mu_\xi.
    \end{array}\right.
\end{align}
\end{proposition}

\begin{remark}\label{remark proposition supersolution grushin model} 
Since $v_\xi$ is even, we can extend $W_\xi$ on $\Omega_x \setminus (-x_\xi,x_\xi)$ by $W_\xi(-x) = W_\xi(x)$, and for every $x \in \Omega_x \setminus (-x_\xi,x_\xi)$, the proposition still holds.
\end{remark}

\begin{remark}\label{remark estimates supersolutions model}
    Due to the upper bound \eqref{upper bound mu_n} on $\mu_\xi$, we have that $x_\xi \rightarrow 0$. Moreover, due to the behaviour of $\mu_\xi$, and $B_\xi$, easy computations show that in the limit at infinity for $|\xi|$, we have for some $C_1, C_2, C_3 > 0$, 
    \begin{align}
\begin{array}{lcl}
    x_\xi &\sim& C_1|\xi|^{-1/(\gamma + 1)} \\
        A_\xi &\sim& C_2|\xi|^{\frac{1}{2(\gamma + 1)}} \\
        B_\xi x_\xi^{\gamma + 1} &\sim& C_3 \text{, since $B_\xi \sim C_3|\xi|$}.
\end{array}
\end{align}
\end{remark}

\subsection{Dissipation speed for the Fourier components of the generalized operator}\label{section spectral G_V real}

We split this section into two parts. Its first part consists in the real spectral analysis of $G_{V,\xi}$, introduced below, which will be of use when the control zone stays at non-negative distance from the singularity in Theorem \ref{th gamma >= 1 euclidean bounded}. The second part consists in the complex spectral analysis, which will be of use for Theorem \ref{th gamma = 1 euclidean koenig}.\\

We introduce, for a generic complex parameter $\xi \in \mathbb{C}$, such that $\operatorname{Re}(\xi) > 0$, and for every $\gamma \geq 1$, the operator $G_{V,\xi}$ on $L^2(\Omega_x)$, defined by 
\begin{align}\label{eqn: definition fourier component generalized grushin bis}
\begin{array}{crl}
D(G_{V,\xi}) &:=& H^2 \cap H_0^1(\Omega_x),  \\[6pt]
G_{V,\xi} u &:=& -u'' + \xi^2q(x)^2 u + V(x)u,
\end{array}
\end{align}
with $q$ and $V$ satisfying the appropriate assumptions of Theorem \ref{th gamma = 1 euclidean koenig} and \ref{th gamma >= 1 euclidean bounded}.\\

We stress that $G_{V,\xi}$ has compact resolvent, and admits an increasing sequence of eigenvalues $\lambda_{k,\xi} \underset{k \rightarrow + \infty}{\longrightarrow} + \infty$ (see \cite{Berezin1991-nn}).

\subsubsection{Real spectral analysis}

In this section, we assume that $\xi \in \mathbb{R}$. Thus, $G_{V,\xi}$ is self-adjoint and its eigenvalues are real.\\

Let us recall the following well-known Lemma for self-adjoint operators.

\begin{lemma}\label{lemma spectral distance}
    Let $A$ be a self-adjoint operator with domain $D(A)$, and $\lambda \in \mathbb{R}$. Then 
    \begin{align}
        \operatorname{dist}(\lambda, \sigma(A)) \leq \frac{\|(A-\lambda)u\|}{\|u\|},
    \end{align}
for every $u \in D(A)$.
\end{lemma}

We can now investigate the behavior of some eigenvalues $\lambda_\xi$ of the operator of interest, that will be shown to behave asymptotically like the $\mu_\xi$'s introduced in Section \ref{section reminders G_0}. 

\begin{proposition}\label{proposition dissipation general grushin bounded}
For every $|\xi| > 0$ large enough, there exists a constant $C >0$ and an eigenvalue $\lambda_\xi$ of $G_{V,\xi}$ that satisfies 
    \begin{align}\label{estimate distance perturbed eigenvalues}
        |\lambda_\xi - \mu_\xi | \leq C|\xi|^{\frac{1}{\gamma+1}}.
    \end{align}
Namely, for $|\xi|$ large enough, there exists $C_2>C_1>0$ such that
\begin{align}\label{estimate eigenvalue perturbed}
    C_1|\xi|^{2/(\gamma + 1)} \leq \lambda_\xi \leq C_2|\xi|^{2/(\gamma + 1)}.
\end{align}
\end{proposition}

\begin{proof}
In the present proof, we shall denote by $C$ any non-negative constant that does not depend on $\xi$, and by $\|\cdot \|$ the $L^2(\Omega_x)$-norm. We denote by $v_\xi$ the first eigenfunction, normalized in $L^2$-norm, of $\mathsf{G}_{\xi}$. We have, applying $G_{V,\xi}$ to $v_\xi$, 
\begin{align*}
    G_{V,\xi} v_\xi = \mu_\xi v_\xi + Q v_\xi + Vv_\xi,
\end{align*}
with 
\begin{align*}
\begin{array}{lcl}
Q v_\xi &=& \left(q(x)^2\xi^2 - \frac{q^{(\gamma)}(0)^2}{(\gamma!)^2}\xi^2 x^{2\gamma}\right) v_\xi.
\end{array} 
\end{align*}
We want to estimate $(G_{V,\xi} - \mu_\xi)v_\xi$ in $L^2$-norm, and the Proposition will follow from Lemma \ref{lemma spectral distance}. Let $\epsilon > 0$ sufficiently small. First, we split the integral as
\begin{align*}
    \|Q v_\xi\|^2 &= \int_{\Omega_x} |Q v_\xi|^2 \ dx \\
    &= \int_{\Omega_x \setminus (-\epsilon,\epsilon)} |Q v_\xi|^2 \ dx + \int_{(-\epsilon,\epsilon)} |Q v_\xi|^2 \ dx .
\end{align*}
We remark that by assumption \ref{H2} on $q$ near the singularity, all the derivatives of $q^2$ up to $2\gamma - 1$ are zero at $x = 0$. Hence, near zero, 
\begin{align*}
    q(x)^2 = \frac{q^{(\gamma)}(0)^2}{(\gamma!)^2}x^{2\gamma} + O\left(|x|^{2\gamma + 1}\right).
\end{align*}

Therefore, by normalization and parity of the $v_\xi$'s, using the supersolutions of Proposition \ref{proposition supersolution grushin model}, and since $x_\xi \rightarrow 0$, for $|\xi|$ large enough we have
\begin{align*}
   \int_{(-\epsilon,\epsilon)} |Q v_\xi|^2 \ dx &\leq  C|\xi|^4 \int_{(-\epsilon,\epsilon)} |x|^{4\gamma + 2}|v_\xi|^2 \ dx \\
   &= C|\xi|^4 \left( \int_0^{x_\xi} |x|^{4\gamma + 2}|v_\xi|^2 \ dx + \int_{x_\xi}^\epsilon |x|^{4\gamma + 2}|v_\xi|^2 \ dx \right)\\
   &\leq C|\xi|^4 \left( x_\xi^{4\gamma + 2} + \int_{x_\xi}^\epsilon |x|^{4\gamma + 2}W_\xi(x)^2 \ dx \right).
\end{align*}

Using the definition of $W_\xi$, and the change of variable $z = x/x_\xi$, we get
\begin{align*}
    \int_{x_\xi}^\epsilon |x|^{4\gamma + 2}W_\xi(x)^2 \ dx &= \int_{x_\xi}^{\epsilon} |x|^{4\gamma + 2} A_\xi^2e^{-2B_\xi x^{\gamma+1}} \ dx \\
    &= \int_{1}^{\epsilon/x_\xi} x_\xi^{4\gamma + 2} |z|^{4\gamma + 2} x_\xi A_\xi^2e^{-2B_\xi x_\xi^{\gamma+1}z^{\gamma+1}} \ dz \\
    &\leq x_\xi^{4\gamma + 2} \int_{1}^{+\infty}  |z|^{4\gamma + 2} x_\xi A_\xi^2e^{-2B_\xi x_\xi^{\gamma+1}z^{\gamma+1}} \ dz.
\end{align*}
Now, thanks to Remark \ref{remark estimates supersolutions model}, for $|\xi|$ large enough, 
\begin{align*}
    C \leq B_\xi x_\xi^{\gamma + 1} \text{ and } \ x_\xi A_\xi^2 \leq C.
\end{align*}
Hence, 
\begin{align*}
    \int_{x_\xi}^{\epsilon} |x|^{4\gamma + 2} A_\xi^2 e^{-2B_\xi x^{\gamma+1}} \ dx \leq C x_\xi^{4\gamma + 2} \int_{1}^{+\infty}  |z|^{4\gamma + 2} e^{-Cz^{\gamma+1}} \ dz.
\end{align*}
with the integral on the right-hand side being finite. We therefore have 
\begin{align}
     \int_{(-\epsilon,\epsilon)} |Q v_\xi|^2 \ dx \leq C|\xi|^4 x_\xi^{4\gamma + 2}.
\end{align}

On the other hand, 
\begin{align*}
\int_{\Omega_x \setminus (-\epsilon,\epsilon)} |Q v_\xi|^2 \ dx &\leq C|\xi|^4\sup_{\Omega_x}\left|q(x)^2- \frac{q^{(\gamma)}(0)^2}{(\gamma!)^2}x^{2\gamma} \right|^2 \int_{\Omega_x \setminus (-\epsilon,\epsilon)} W_\xi(x)^2 \ dx \\
&= C|\xi|^4 \int_{\epsilon}^1 A_\xi^2e^{-2B_\xi x^{\gamma+1}} \ dx \\
&\leq C|\xi|^4 A_\xi^2e^{-2B_\xi\epsilon^{\gamma+1}}.
\end{align*}
Obviously, by Remark \ref{remark estimates supersolutions model}, the right-hand term decays exponentially fast as $|\xi| \rightarrow +\infty$. Therefore, there exists a constant $C>0$ such that 
\begin{align}
    \|Q v_\xi\| \leq C\xi^2 x_\xi^{2\gamma + 1}.
\end{align}
Using once again Remark \ref{remark estimates supersolutions model}, we have that 
\begin{align*}
    \xi^2 x_\xi^{2\gamma + 1} \leq C \xi^{2 - \frac{2\gamma +1}{\gamma + 1}} = C \xi^{\frac{1}{\gamma + 1}}.
\end{align*}
Finally, we obviously have
\begin{align*}
    \|Vv_\xi\| \leq \sup_{\Omega_x} |V(x)| < \infty,
\end{align*}
and \eqref{estimate distance perturbed eigenvalues} follows. Coupling \eqref{estimate distance perturbed eigenvalues} with Proposition \ref{proposition dissipation classical Omega bounded}, we directly obtain \eqref{estimate eigenvalue perturbed}.  
\end{proof}

\subsubsection{Complex spectral analysis}\label{section spectral G_V complex}

In this subsection, we show that in the case $\gamma = 1$, we can obtain more precise estimates on $\lambda_\xi$. The results here directly follow from the spectral analysis provided in \cite{darde2023null}, in the case $V = 0$. As a matter of fact, we show that the spectral analysis from \cite{darde2023null} still holds for our operator.\\

We first need to reintroduce some objects from \cite{darde2023null}, for which we keep as much as possible the same notations. Set $I := (-L,L)$. Denote by $\mathbbm{1}_I$ the operator that maps $v \in L^2(I)$ to its extension by zero on $\mathbb{R}$, and by $\mathbbm{1}^*_I$ its adjoint, which maps $v \in L^2(\mathbb{R})$ to its restriction on $I$. Set
\begin{align}
    \Sigma_{\theta_0} := \{ \nu \in \mathbb{C}, |\nu| \geq 1, \arg(\nu) \leq \theta_0 \} \text{, for some $\theta_0 \in [0, \pi/2)$. }
\end{align}

Let $R > 0$ and $\epsilon > 0$. For $\nu \in \Sigma_{\theta_0}$, set
\begin{align}
    \Tilde{\mathcal{Z}}_\nu = \mathcal{Z}_{q'(0)\nu,R,\epsilon} := \left\{ z \in \mathbb{C}, \ |z| \leq R q'(0)|\nu|, \ \text{dist}(z,\sigma(H_{q'(0)\nu})) \geq \epsilon q'(0)|\nu| \right\} .
\end{align}
We shall write $G_{V,\nu}$, $\nu \in \mathbb{C}$, to be $G_{V,\xi}$ with $\xi$ complex, to keep up with the notations of \cite{darde2023null}, and $G_\nu$ to be $G_{V,\nu}$ with $V = 0$, not to be mistaken with the classical Grushin operator. Also, we denote, for $\beta \in \mathbb{C}$ with $Re(\beta) > 0$, the non-self-adjoint harmonic oscillator $H_\beta = - \partial_x^2 + \beta^2 x^2 $ on $\mathbb{R}$. 

\begin{proposition}\label{proposition estimates resolvent}
There exists $\nu_0 \geq 1$ such that for every $\nu \in \Sigma_{\theta_0}$, $|\nu| \geq \nu_0$ and $z \in \Tilde{\mathcal{Z}}_\nu$, we have $z \in \rho(G_{V,\nu})$, and
\begin{align}
    \sup_{z \in  \Tilde{\mathcal{Z}}_\nu } \left|\left| (G_{V,\nu} - z)^{-1} - \mathbbm{1}^*_I(H_{q'(0)\nu}-z)^{-1} \mathbbm{1}_I  \right|\right| = \underset{\underset{\nu \in \Sigma_{\theta_0}}{|\nu| \rightarrow + \infty}}{O}\left(\frac{1}{|\nu|}\right).
\end{align}
\end{proposition}

\begin{proof}
In the present proof, we shall denote by $\|\cdot \|$ the operator norm. From \cite[Proposition 4.1]{darde2023null}, we know that such a $z$ introduced in the proposition is in fact in $\rho(G_\nu)$. For such a $z$, we write
\begin{align*}
   G_{V,\nu} - z = G_\nu + V - z = [1 + V(G_\nu - z)^{-1}](G_\nu - z),
\end{align*}
where the potential $V$ can be understood as a bounded linear operator on $L^2(I)$. Since we know from \cite[Proposition 4.1]{darde2023null} combined with \cite[Proposition C.1 (iv)]{darde2023null}, that 
\begin{align}\label{estimate resolvent koenig}
    \sup_{z \in  \Tilde{\mathcal{Z}}_\nu } \|(G_\nu - z)^{-1}\| = \underset{\underset{\nu \in \Sigma_{\theta_0}}{|\nu| \rightarrow + \infty}}{O}\left(\frac{1}{|\nu|}\right),
\end{align}
it follows that $1 + V(G_\nu - z)^{-1}$ is an invertible linear operator on $L^2(I)$ for $|\nu|$ large enough, and so is $G_{V,\nu} - z$. Moreover,
\begin{align*}
    \|[1 + V(G_\nu - z)^{-1}]^{-1}\| = \frac{1}{1 - \|V(G_\nu - z)^{-1}\|} \rightarrow 1,
\end{align*}
as $|\nu| \rightarrow + \infty$, $\nu \in \Sigma_{\theta_0}$, by \eqref{estimate resolvent koenig}. Hence, since
\begin{align*}
   (G_{V,\nu} - z )^{-1} = (G_\nu - z)^{-1}[1 + V(G_\nu - z)^{-1}]^{-1}, 
\end{align*}
it follows that we also have  
\begin{align}\label{estimate resolvent roman}
    \sup_{z \in  \Tilde{\mathcal{Z}}_\nu }\|(G_{V,\nu} - z )^{-1} \| = \underset{\underset{\nu \in \Sigma_{\theta_0}}{|\nu| \rightarrow + \infty}}{O}\left(\frac{1}{|\nu|}\right).
\end{align}
By triangular inequality, since we already know from \cite[Proposition 4.1]{darde2023null} that
\begin{align*}
 \sup_{z \in  \Tilde{\mathcal{Z}} } \left|\left| (G_\nu - z)^{-1} - \mathbbm{1}^*_I(H_{q'(0)\nu}-z)^{-1} \mathbbm{1}_I  \right|\right| = \underset{\underset{\nu \in \Sigma_{\theta_0}}{|\nu| \rightarrow + \infty}}{o}\left(\frac{1}{|\nu|}\right),
\end{align*}
it is sufficient to check that 
\begin{align*}
    \sup_{z \in  \Tilde{\mathcal{Z}}_\nu } \left|\left| (G_{V,\nu} - z)^{-1} - (G_\nu - z)^{-1} \mathbbm{1}_I  \right|\right| = \underset{\underset{\nu \in \Sigma_{\theta_0}}{|\nu| \rightarrow + \infty}}{O}\left(\frac{1}{|\nu|}\right),
\end{align*}
but this is a direct consequence of \eqref{estimate resolvent koenig} and \eqref{estimate resolvent roman}, and the proposition follows.
\end{proof}

Now, using basic tools from perturbation theory (see for instance \cite{Kato1995-xh}), or following the exact same proofs of \cite[Proposition 3.6, Proposition 4.2]{darde2023null}, we get the following corollary.

\begin{corollary}\label{corollary dissipation general grushin}
In the limit $|\xi| \rightarrow +\infty$, $\xi \in \Sigma_{\theta_0}$, we have that $\lambda_\xi$ is algebraically simple, and 
\begin{align}
    \lambda_\xi = q'(0) \xi + o(|\xi|).
\end{align}
\end{corollary}

\section{The Grushin equation on tensorized domains}\label{section grushin euclidiens}

In this section, we prove the non null-controllability of our generalized Grushin operators for various considerations of the control zone in some Euclidean domains. 

\subsection{Proof of Theorem \ref{th gamma >= 1 euclidean bounded}: control on vertical strips}\label{section proof th euclid bounded}

\subsubsection{Agmon estimates}

Assume that $\Omega_y = \mathbb{S}^1 \text{ or } (0,\pi)$. In this case, we consider a sequence of solutions of the form
\begin{align*}
    g_n(t,x,y) = e^{-\lambda_n t}v_n(x)\phi_n(y),
\end{align*}
for $n$ large, where $\phi_n$ is a normalized in norm eigenfunction of $-\partial_y(r(y)^2\partial_y)$ on $\Omega_y$, associated to $\xi_n^2$, and $v_n$ a normalized in norm eigenfunction of $G_n := G_{V,\xi_n}$ defined in \eqref{eqn: definition fourier component generalized grushin bis}, associated to $\lambda_n$ that satisfies the estimate of Proposition \ref{proposition dissipation general grushin bounded} when $\gamma \geq 1$, and Corollary \ref{corollary dissipation general grushin} when $\gamma = 1$. We abuse the notation here by denoting $\lambda_n$ and $v_n$ instead of $\lambda_{\xi_n}$ and $v_{\xi_n}$. We write the control zone as $\omega := \omega_x \times \Omega_y$, with $\overline{\omega_x} \cap \{x=0\} = \emptyset$.\\

Introduce the set 
\begin{align}\label{eqn: definition F delta}
    F_\delta := \left\{ x \in \Omega_x, \quad q(x)^2 \leq \delta \right\}.
\end{align}
By assumption \ref{H2}, there exists $\delta_0$, such that for every $0 \leq \delta < \delta_0$, the set $F_\delta$ is a connected neighborhood of $0$.\\

Introduce the Agmon distance between $\omega_x$ and $F_\delta$, as 
\begin{align}\label{eqn: agmon distance between sets}
    d_{\operatorname{agm},\delta}(\omega_x,F_\delta) = \inf_{x \in \omega_x, y \in F_\delta} \int_{\min(x,y)}^{\max(x,y)} (q(s)^2 - \delta)_+^{1/2} \ ds. 
\end{align}
Observe that if $q$ satisfies \ref{H2}, then as $\delta \rightarrow 0$ we retrieve the aforementioned Agmon distance \eqref{agmon distance gamma = 1}.\\

The proof of Theorem \ref{th gamma >= 1 euclidean bounded} follows directly from the following proposition for which we follow the ideas of Agmon estimates (see e.g. \cite{Helffer1988-yw, Agmon}).

\begin{proposition}\label{proposition agmon decay grushin general}
There exists $\delta_0 > 0$, such that for every $\delta \in (0,\delta_0)$, there exist $n_0 \in \mathbb{N}$ and $C > 0$, such that for every $n \geq n_0$,
\begin{align}
    \int_{\omega_x}  v_n(x)^2 \ dx \leq C e^{-2\xi_n(1-\delta)d_{\operatorname{agm},\delta}(\omega_x,F_\delta)},
\end{align}
where $F_\delta$ is defined by \eqref{eqn: definition F delta} and satisfies $F_\delta \cap \overline{{\omega}_x} = \emptyset$, and $d_{\operatorname{agm},\delta}(\omega_x,F_\delta)$ is defined by \eqref{eqn: agmon distance between sets}.

\end{proposition}

\begin{proof} We drop every notation concerning the variable $x$ to simplify the reading. Let $v_n = \theta_ne^{-\xi_n\psi}$ for some sufficiently regular function $\psi$. Then, we have
\begin{align*}
    v_n' &= \theta_n'e^{-\xi_n\psi} - \xi_n \psi'\theta_ne^{-\xi_n\psi}, \\
    v_n'' &= \theta_n''e^{-\xi_n\psi} - 2\xi_n\theta_n'\psi'e^{-\xi_n\psi} - \xi_n\psi''\theta_ne^{-\xi_n\psi} + \xi_n^2\psi'^2\theta_ne^{-\xi_n\psi}.
\end{align*}

Since $(G_n - \lambda_n)v_n = 0$, we get that $\theta_n$ must satisfy 
\begin{align}\label{eqn: agmon equality}
    -\theta_n''+ 2\xi_n\theta_n'\psi' + (\xi_n\psi'' - \xi_n^2\psi'^2 + \xi_n^2q^2 + V - \lambda_n) \theta_n = 0.
\end{align}
Integrating by parts, we first remark that 
\begin{align*}
   \int_{\Omega_x} \theta_n'\psi'\theta_n = \frac{1}{2}\int_{\Omega_x} \frac{d}{dx}(\theta_n^2) \psi',
\end{align*}
which implies that 
\begin{align*}
     \int_{\Omega_x} \theta_n'\psi'\theta_n = -\frac{1}{2} \int_{\Omega_x} \psi''\theta_n^2.
\end{align*}
Therefore, multiplying \eqref{eqn: agmon equality} by $\theta_n$ and integrating by parts, we get
\begin{align*}
   \int_{\Omega_x} \theta_n'^2+ (\xi_n^2 q^2 + V - \xi_n^2\psi'^2 - \lambda_n) \theta_n^2 = 0,
\end{align*}
which implies, dividing by $\xi_n^2$,
\begin{align*}
     \frac{1}{\xi_n^2} \int_{\Omega_x} \theta_n'^2 + \int_{\Omega_x} \left(q^2 - \psi'^2 + \frac{V}{\xi_n^2} - \frac{\lambda_n}{\xi_n^2} \right) \theta_n^2 = 0,
\end{align*}
and
\begin{align*}
    \int_{\Omega_x} \left(q^2 - \psi'^2 + \frac{V}{\xi_n^2} - \frac{\lambda_n}{\xi_n^2} \right) \theta_n^2 \leq 0.
\end{align*}

Since the only zero for the function $q$ is at $x=0$ by assumption \ref{H2}, there exists $\delta_0 > 0$ sufficiently small such that for every $\delta \in (0,\delta_0)$ the set $F_\delta := \{ x \in \Omega_x, q(x)^2 \leq \delta \}$ is a connected neighborhood of zero, which is decreasing for the inclusion as $\delta \rightarrow 0^+$, and moreover satisfies $F_\delta \cap \overline{\omega_x} = \emptyset$ since $\overline{\omega_x} \cap \{0\} = \emptyset$. We choose such a $\delta \in (0,\delta_0)$. \\

As we know that $\lambda_n/ \xi_n^2$ tends to zero as $n$ tends to infinity thanks to Proposition \ref{proposition dissipation general grushin bounded}, and the fact that $V \in L^\infty(\Omega)$, there exists $n_0 \in \mathbb{N}$ such that for every $n \geq n_0$ we have
\begin{align*}
    \left\{ x \in \Omega_x, q(x)^2 \leq \frac{\lambda_n - V}{\xi_n^2} \right\} \subset F_\delta.
\end{align*}

Thus, for every $n \geq n_0$ we get
\begin{align*}
    \int_{\Omega_x} (q^2 - \psi'^2 - \delta) \theta_n^2 \leq 0.
\end{align*}
Choose $\psi(x) := (1-\delta)d_{\operatorname{agm},\delta}(x,F_\delta)$, where $d_{\operatorname{agm}}$ is the degenerated distance induced by the degenerated Agmon metric $(q^2 - \delta)_+ dx^2$, where $dx^2$ is the standard Euclidean metric on $\Omega_x$. We note that we have 
\begin{align}
    |\psi'(x)|^2 \leq (1-\delta)^2(q^2 - \delta)_+.
\end{align}
Therefore, on the set $F_\delta$ we have that $q^2 - \psi'^2 - \delta = q^2 - \delta \leq 0$, and on the complement of $F_\delta$ we have 
\begin{align*}
    q^2 - \psi'^2 - \delta &\geq (1-(1-\delta)^2)(q^2 - \delta)_+ \\
    &= (2\delta - \delta^2)(q^2 - \delta)_+ \\
    &\geq 0.
\end{align*}
It follows that 
\begin{align*}
    (2\delta - \delta^2) \int_{\Omega_x \setminus F_\delta} (q^2 - \delta) \theta_n^2 \leq - \int_{F_\delta} (q^2 - \delta) \theta_n^2.
\end{align*}

Replacing $\theta_n$ by its expression $v_n(x)e^{\xi_n(1-\delta)d_{\operatorname{agm},\delta}(x,F_\delta)}$, using the fact that $\overline{\omega_x} \cap F_\delta = \emptyset$, and that $d_{\operatorname{agm},\delta}(x,F_\delta) = 0$ on $F_\delta$, we get that on one hand
\begin{align*}
    - \int_{F_\delta}  (q^2 - \delta) \theta_n^2 &=  - \int_{F_\delta}  (q^2 - \delta) v_n^2 \\
    &\leq \sup_{F_\delta}|q^2 - \delta|,
\end{align*}
since $v_n$ is normalized in $L^2$-norm. On the other hand,
\begin{align*}
 (2\delta - \delta^2) \int_{\Omega_x \setminus F_\delta} (q^2 - \delta) \theta_n^2  &=  (2\delta - \delta^2) \int_{\Omega_x \setminus F_\delta} (q^2 - \delta) v_n^2  e^{2\xi_n(1-\delta)d_{\operatorname{agm},\delta}(x,F_\delta)} \\
 &\geq (2\delta - \delta^2) \int_{\omega_x} (q^2 - \delta) v_n^2  e^{2\xi_n(1-\delta)d_{\operatorname{agm},\delta}(x,F_\delta)} \\
 &\geq (2\delta - \delta^2) \min_{\omega_x} (q^2 - \delta) e^{2\xi_n(1-\delta)d_{\operatorname{agm},\delta}(\omega_x,F_\delta)} \int_{\omega_x}  v_n^2,
\end{align*}
where we stress that $\min_{\omega_x} (q^2 - \delta) > 0$. Therefore, we finally get
\begin{align*}
   \int_{\omega_x}  v_n^2 \leq \frac{\sup_{F_\delta}|q^2 - \delta|}{(2\delta - \delta^2) \min_{\omega_x} (q^2 - \delta)} e^{-2\xi_n(1-\delta)d_{\operatorname{agm},\delta}(\omega_x,F_\delta)},
\end{align*}
which proves the proposition. 
\end{proof}

\subsubsection{Proof of Theorem \ref{th gamma >= 1 euclidean bounded}}

We can now proceed to the proof of Theorem \ref{th gamma >= 1 euclidean bounded}.

\begin{proof}[Proof of theorem \ref{th gamma >= 1 euclidean bounded} in the case $\Omega_y = \mathbb{S}^1 \text{ or } (0,\pi)$] Let us split the proof in two parts.\\

\underline{Proof in the case $\gamma > 1$}\\

Choose the sequence of solutions $g_n$ introduced at the beginning of the section. Assume that system \eqref{adjoint system change variable} is observable from $\omega$ in time $T > 0$. Then, there exists $C > 0$ such that for any $n \in \mathbb{N}^*$,  
\begin{align*}
     \frac{e^{-2\lambda_n T}}{\lambda_n}\int_{\Omega_x} v_n(x)^2 \ dx &\leq C \int_0^T \int_{\omega_x} e^{-2\lambda_n t} v_n(x)^2 \ dx \\
     &\leq C T  \int_{\omega_x} v_n(x)^2 \ dx.
\end{align*}
Thanks to Proposition \ref{proposition dissipation general grushin bounded}, we can apply Proposition \ref{proposition agmon decay grushin general} that states that there exists two constants $C_1,C_2 > 0$, such that for every $n$ sufficiently large, since the integral in the left-hand side values one, 
\begin{align*}
    1 \leq C_1CT\lambda_ne^{2\lambda_n T}e^{-\xi_nC_2}.
\end{align*}

Using now the fact that $\lambda_n = \mathcal{O}(n^{\frac{2}{\gamma + 1}})$ by Proposition \ref{proposition dissipation general grushin bounded}, we have that for some $C_3 > 0$
\begin{align*}
    1 \leq C_1C_3CTn^{\frac{2}{\gamma + 1}}e^{2C_3\xi_n^{\frac{2}{\gamma + 1}}T}e^{-\xi_n C_2}.
\end{align*}
We therefore simply have to show that for any time $T > 0$, the right hand side tends to zero as $n$ tends to infinity. This is straightforward since 
\begin{align*}
2C_3\xi_n^{\frac{2}{\gamma + 1}}T - \xi_n C_2 = \xi_n(\xi_n^{\frac{1 - \gamma}{ \gamma + 1}}2C_3T - C_2),
\end{align*}
becomes negative for any time $T > 0$ as $n$ tends to infinity since $\gamma > 1$. This disproves the observability inequality for any time $T > 0$. \\

\underline{Proof in the case $\gamma = 1$.}\\

In this case, the above strategy still works to disprove the observability in small time, and give analogous results as the ones known from \cite{beauchard2014null, darde2023null}. Similarly to the previous proof, assuming that system \eqref{adjoint system change variable} is observable in time $T > 0$ from $\omega$ implies, applying Proposition \ref{proposition agmon decay grushin general}, that for $n$ large enough and some constants $C,C_1 > 0$,
\begin{align*}
    1 \leq C_1CT \lambda_n e^{2\lambda_nT}e^{-2\xi_n(1-\delta)d_{\operatorname{agm},\delta}(\omega_x,F_\delta)},
\end{align*}
with $\delta > 0$ an arbitrary small parameter. Thus, for any $T > 0$, observability cannot hold if
\begin{align*}
    e^{2\lambda_n T - 2\xi_n(1-\delta)d_{\operatorname{agm},\delta}(\omega_x,F_\delta)} \rightarrow 0, \text{ as $n$ tends to infinity.}
\end{align*}
We use that $\lambda_n = q'(0)\xi_n + o(\xi_n)$ for large values of $n$ from Corollary \ref{corollary dissipation general grushin}. Letting $\epsilon > 0$, for $n$ large enough we have $\lambda_n \leq (1+\epsilon)\xi_n$, and from the above estimate, the observability inequality cannot hold if, for $n$ large,  
\begin{align*}
    2(1+\epsilon)q'(0)\xi_n T - 2\xi_n(1-\delta)d_{\operatorname{agm},\delta}(\omega_x,F_\delta) < 0.
\end{align*}
That is, we cannot have
\begin{align*}
    T <  \frac{(1-\delta)d_{\operatorname{agm},\delta}(\omega_x,F_\delta)}{(1+\epsilon)q'(0)}.
\end{align*}
Since this is true for any $\delta,\epsilon > 0$ arbitrary small, we cannot have 
\begin{align*}
    T <  \frac{d_{\operatorname{agm}}(\omega_x)}{q'(0)}.
\end{align*}

This concludes the proof in the case $\gamma = 1$.
\end{proof}

Let us now proceed to the case of $\Omega_y = \mathbb{R}$ and show that it is analogous to the previous case. 

\begin{proof}[Proof of theorem \ref{th gamma >= 1 euclidean bounded} in the case $\Omega_y = \mathbb{R}$]
Recall that we set $r(y) = 1$ on $\mathbb{R}$. Due to the absence of restriction in the $y$-direction for some $x$, the (equivalent) problems of null-controllability and observability are equivalent to their formulation in Fourier components. Namely, \eqref{observability inequality Omega} is equivalent to
\begin{align}\label{observability inequality Fourier}
\|\hat{g}(T)\|_{D\left(G_{V,\xi}^{-1/2}\right)}^2  \leq C \int_0^T \int_\mathbb{R} \int_{\Tilde{\omega}_x} |\hat{g}(t,x,\xi)|^2 \ dx \ d\xi \ dt,
\end{align}
with $\hat{g}$ the solution of 
\begin{align}\label{grushin system adjoint fourier}
    \left\{ \begin{array}{lcll}
     \partial_t \hat{g} - \partial_x^2\hat{g} + \xi^2 q(x)^2\hat{g} + V(x)\hat{g} & = & 0, & (t,x,\xi) \in (0,T) \times (-L,L) \times \mathbb{R},\\
     \hat{g}(t,\pm L,\xi) &=& 0, & (t,y) \in (0,T) \times \mathbb{R}, \\
     \hat{g}(0,x,\xi) & = & \hat{g}_0(x,\xi), & (x,\xi) \in (-L,L) \times \mathbb{R},
    \end{array} \right.
\end{align}
where $\hat{g}$ is the partial Fourier transform of $g$ with respect to $y$, and $q$ satisfies assumption \ref{H2}. \\

As we do not have to come back to the functions in the $y$ variable, we can drop the hat notation. We also write the dependence in $\xi$ as a variable and not as a subscript. By Proposition \ref{proposition dissipation general grushin bounded}, we can extract an eigenvalue of 
\begin{align}\label{fourier eigenfunction}
    G_{V,\xi} = - \partial_x^2 + q(x)^2\xi^2 + V(x),
\end{align}
that we denote by $\lambda(\xi)$, and which satisfies for $|\xi| >0 $ sufficiently large
\begin{align}
    \lambda(\xi) \leq C|\xi|^{2/(\gamma+1)}, \quad \text{for some } C >0.
\end{align}
In the case where $q$ satisfies \ref{H2} with $\gamma=1$, we have from Corollary \ref{corollary dissipation general grushin} that $\lambda(\xi)$ can be chosen as the first eigenvalue of $G_{V,\xi}$ \eqref{fourier eigenfunction} which satisfies in the limit $\xi \rightarrow + \infty$
\begin{align}
    \lambda(\xi) = \xi q'(0) + o(|\xi|).
\end{align}
We denote by $g(\cdot,\xi)$ that associated eigenfunction normalized in $L^2(\Omega_x)$-norm.\\

Due to the form of \eqref{grushin system adjoint fourier}, we can localize these eigenfunctions around high frequencies by multiplying them by a smooth compactly supported function. Let $\delta > 0$, and set, for $n>0$, $\psi_n \in C_c^\infty(\mathbb{R})$ verifying
\begin{align}
    \operatorname{supp}(\psi_n) = [n-\delta, n+\delta], \qquad 0 \leq \psi_n \leq 1, \qquad \psi_n(x) = 1, \ \forall x \in [n-\delta/2, n+\delta/2].
\end{align}
Hence, the sequence of functions
\begin{align}
    g_n(t,x,\xi) = e^{-\lambda(\xi)t}g(x,\xi)\psi_n(\xi)
\end{align}
are solutions of \eqref{grushin system adjoint fourier}. Plugging these solutions into the observability inequality \eqref{observability inequality Fourier}, we get 
\begin{align*}
\int_{\mathbb{R}^2} \frac{e^{-2\lambda(\xi)T}}{\lambda(\xi)} g(x,\xi)^2\psi_n(\xi)^2 \ dx \ d\xi \leq C \int_0^T \int_\mathbb{R} \int_{\Tilde{\omega}_x} e^{-2\lambda(\xi)t}g(x,\xi)^2\psi_n(\xi)^2 \ dx \ d\xi \ dt.
\end{align*}
By normalization of $g(\xi, \cdot)$, for $n$ large enough, the left-hand term is bounded below by
\begin{align*}
\int_{\mathbb{R}^2} \frac{e^{-2\lambda(\xi)T}}{\lambda(\xi)} g(x,\xi)^2\psi_n(\xi)^2 \ dx \ d\xi &\geq  \int_{n-\delta/2}^{n+\delta/2} \int_\mathbb{R} \frac{e^{-2\lambda(\xi)T}}{\lambda(\xi)}g(x,\xi)^2 \ dx \ d\xi \\
&= \int_{n-\delta/2}^{n+\delta/2} \frac{e^{-2\lambda(\xi)T}}{\lambda(\xi)} d\xi.
\end{align*}
For the right-hand term, for $n$ large enough, we have
\begin{align*}
\int_0^T \int_\mathbb{R} \int_{\Tilde{\omega}_x} e^{-2\lambda(\xi)t}g(x,\xi)^2\psi_n(\xi)^2 \ dx \ d\xi \ dt &\leq \int_0^T \int_{n-\delta}^{n+\delta} \int_{\Tilde{\omega}_x} e^{-2\lambda(\xi)t} g(x,\xi)^2\ dx \ d\xi \\
&\leq T \int_{n-\delta}^{n+\delta} \int_{\Tilde{\omega}_x} g(x,\xi)^2\ dx \ d\xi. 
\end{align*}

Notice that Proposition \ref{proposition agmon decay grushin general} holds replacing $\xi_n$ with a generic parameter $\xi$, for $|\xi|$ large enough. Thus, we can keep bounding the right-hand side above using Proposition \ref{proposition agmon decay grushin general}. We will use the same parameter $\delta > 0$ arbitrary small in both the definition of $\psi_n$ and when applying Proposition \ref{proposition agmon decay grushin general} and Corollary \ref{corollary dissipation general grushin}.\\

\underline{Proof in the case $\gamma > 1$.}\\

In this case, we obtain for some constants $C_0,C_1,C_2 > 0$ and $n$ sufficiently large, on one hand
\begin{align*}
  \int_{\mathbb{R}^2} \frac{e^{-2\lambda(\xi)T}}{\lambda(\xi)} g(x,\xi)^2\psi_n(\xi)^2 \ dx \ d\xi  &\geq \int_{n-\delta/2}^{n+\delta/2}\frac{e^{-2\lambda(\xi)T}}{\lambda(\xi)}  d\xi \\
  &\geq \frac{e^{-2C_0(n+\delta/2)^{2/(\gamma +1)}T}}{\lambda(n+\delta/2)},
\end{align*}
and on the other hand,
\begin{align*}
    \int_0^T \int_\mathbb{R} \int_{\Tilde{\omega}_x} e^{-2\lambda(\xi)t}g(x,\xi)^2\psi_n(\xi)^2 \ dx \ d\xi \ dt &\leq C_1T \int_{n-\delta}^{n+\delta} e^{-\xi C_2} \ d\xi \\
    &\leq 2\delta C_1 Te^{-(n-\delta)C_2}. 
\end{align*}
Thus, observability cannot hold in any time $T > 0$ such that 
\begin{align}
    2C_0(n+\delta/2)^{2/(\gamma+1)}T-(n-\delta)C_2 < 0, \quad \text{as } n \rightarrow + \infty.
\end{align}
But since $\gamma >1$, we observe that the above becomes negative in the limit $n \rightarrow + \infty$ for any fixed time $T > 0$. Thus, observability never occurs.\\

\underline{Proof in the case $\gamma = 1$.} \\

In this case, we have on one hand, for $n$ sufficiently large, by Corollary \ref{corollary dissipation general grushin},
\begin{align*}
  \int_{\mathbb{R}^2} \frac{e^{-2\lambda(\xi)T}}{\lambda(\xi)} g(x,\xi)^2\psi_n(\xi)^2 \ dx \ d\xi  &\geq \int_{n-\delta/2}^{n+\delta/2}\frac{e^{-2\lambda(\xi)T}}{\lambda(\xi)}  d\xi \\
  &\geq \frac{e^{-2q'(0)(1+\delta)(n+\delta/2)T}}{\lambda(n + \delta/2)}.
\end{align*}
On the other hand, by Proposition \ref{proposition agmon decay grushin general},
\begin{align*}
    \int_0^T \int_\mathbb{R} \int_{\Tilde{\omega}_x} e^{-2\lambda(\xi)t}g(x,\xi)^2\psi_n(\xi)^2 \ dx \ d\xi \ dt &\leq C(\delta)T \int_{n-\delta}^{n+\delta} e^{-2\xi(1-\delta)d_{\operatorname{agm}}(\tilde{\omega}_x,F_\delta)} \ d\xi \\
    &\leq 2\delta C(\delta) Te^{-2(n+\delta)(1-\delta)d_{\operatorname{agm}}(\tilde{\omega}_x,F_\delta)}. 
\end{align*}
Thus, observability cannot hold in any time $T > 0$ such that 
\begin{align}
    \displaystyle 2q'(0)(1+\delta)(n+\delta/2)T - 2(n+\delta)(1-\delta)d_{\operatorname{agm}}(\tilde{\omega}_x,F_\delta) < 0, \quad \text{as } n \rightarrow + \infty.
\end{align}
It is therefore not hard to see that this cannot hold in any time 
\begin{align}
    T < \frac{1-\delta}{q'(0)(1+\delta)}d_{\operatorname{agm}}(\tilde{\omega}_x,F_\delta).
\end{align}

Since $\delta > 0$ can be taken arbitrary small, this yields the sought result and concludes the proof. 
\end{proof}

\subsection{Proof of Theorem \ref{th gamma = 1 euclidean koenig}: control in the complement of a rectangle}\label{section proof th euclid bounded koenig}

The proof of this theorem is the same as \cite[Theorem 1.3]{darde2023null} in the case $V =0$ and for initial states in $L^2(\Omega)$, up to some remarks. We therefore only state the modifications to take into account, and omit some details. \\

We set ourselves on $L^2(\Omega, \mathbb{C})$, for which non null-controllability implies the same on $L^2(\Omega, \mathbb{R})$. Indeed, the Grushin operator $G_V$ having real-valued coefficients, it acts on complex-valued functions as $\Hat{G}_V(f+ig) = G_Vf + iG_Vg$. It inherits all the properties of $G_V$, has domain $D(\Hat{G}_V) = D(G_V) + iD(G_V)$, and generates a continuous semigroup $(T(t))_t$, which satisfies, denoting $(S(t))_t$ the one of $G_V$, $T(t)(f+ig) = S(t)f + iS(t)g$. Hence, it is not hard to see that observability in time $T > 0$ for $G_V$ implies the same for $\Hat{G}_V$. In particular, given a sequence of solutions refuting the observability for $\Hat{G}_V$, the real-valued functions given by either the real part or imaginary part will disprove the observability on $L^2(\Omega, \mathbb{R})$ for $G_V$. We shall drop the hat notation for the complexification of $G_V$ has there shall not be any confusion. 

\begin{proof}[Proof of Theorem \ref{th gamma = 1 euclidean koenig}]
We prove the theorem in the case $\Omega_y = \mathbb{S}^1$. The case $(0,\pi)$ then follows from \cite[Appendix A]{darde2023null}. Given a sequence of complex numbers $(a_n)_{n}$, a solution of system \eqref{control system change variable} is given by
\begin{align}
    g(t,x,y) = \sum_{n > N} a_{n-1} g(t,x,n)e^{iny},
\end{align}
where the sum is taken finite, $N \in \mathbb{N}$ is taken sufficiently large as specified later, and $g(t,x,n)$ is the solution of system 
\begin{align} \label{adjoint system change variable fourier}
\left\{ \begin{array}{lcll}
     \left( \partial_t - \partial_x^2 + n^2 q(x)^2 + V(x) \right) g & = & 0, & (t,x) \in (0,T) \times (-L,L),\\
     g(t,x,n) & = & 0, & (t,x) \in (0,T) \times \partial (-L,L), \\
     g(0,x,n) & = & g_0(x,n), & x \in (-L,L).
    \end{array} \right.
\end{align}
Denote by $v_n$ an eigenfunction associated to the first eigenvalue $\lambda_n$ of the operator 
$G_{n,V} = - \partial_x^2 + n^2 q(x)^2 + V(x)$. Then $g(t,x,n) = e^{-\lambda_n t}v_n(x)$. More precisely, we choose $v_n$ as follows. \\

Recall that we denote, for $\beta \in \mathbb{C}$ with $Re(\beta) > 0$, the non-self-adjoint harmonic oscillator $H_\beta = - \partial_x^2 + \beta^2 x^2 $ on $\mathbb{R}$. Then $G_{n,V}$ is a perturbation of $H_{q'(0)n}$ restricted to $(-L,L)$. \\

Set $\Tilde{v}_n(x) = n^{1/4}e^{-nq'(0)x^2/2}$ the first eigenfunction of $H_{q'(0)n}$. We choose $v_n$ as  
\begin{align}\label{eqn: def vn spectral proj}
    v_n := \Pi_n \Tilde{v}_n,
\end{align}
where $\Pi_n$ is the spectral projector onto the first eigenspace associated to $\lambda_n$ of $G_{n,V}$. Our solution $g$ now writes as 
\begin{align}\label{eqn: form solution complement rectangle}
    g(t,x,y) = \sum_{n > N} e^{-\lambda_n t} a_{n-1} v_n(x)e^{iny}.
\end{align}

Assume now that observability \eqref{observability inequality Omega} holds in time $T > 0$ from $\tilde{\omega}$. \\

We plug our solutions \eqref{eqn: form solution complement rectangle}, with $v_n$ defined by \eqref{eqn: def vn spectral proj}, into \eqref{observability inequality Omega}, to obtain that there exists $C > 0$, such for any solution of the form \eqref{eqn: form solution complement rectangle} we have 
\begin{align}\label{eqn: obs form solution complement rectangle}
\int_\Omega |g(T,x,y)|^2 \ dx \ dy \leq C \int_0^T \int_{\Tilde{\omega}} |g(t,x,y)|^2 \ dx \ dy \ dt.
\end{align}

We will show that \eqref{eqn: obs form solution complement rectangle} implies an $L^2-L^\infty$ estimates on complex polynomials $p(z) = \sum_{n \geq N}a_nz^n$ with a zero of order $N$ at zero. \\

Let us first treat the left-hand side of \eqref{eqn: obs form solution complement rectangle}, showing that it can be bounded from below by some $L^2$-norm of complex polynomials, following the proof of \cite[Lemma 3.7]{darde2023null}.\\

Thanks to Proposition \ref{proposition estimates resolvent}, it follows from \cite[Proposition 3.6 and Proposition 4.2]{darde2023null} that $\|v_n\|_{L^2(\Omega_x)}$ is bounded below uniformly in $n$. Thus, thanks to Corollary \ref{corollary dissipation general grushin} from which we have $\lambda_n \sim nq'(0)$, and using the uniform lower bound on $\|v_n\|_{L^2(\Omega_x)}$, for $\epsilon > 0$, for some constants $c,C_\epsilon> 0$, and for $N > 0$ sufficiently large, we have
\begin{align*}
     \|g(T)\|_{D\left(G_V^{-1/2}\right)}^2 \ dx \ dy &=  \sum_{n > N} \frac{|a_{n-1}|^2}{\lambda_n} \| v_n \|_{L^2(\Omega_x)}^2 e^{-2\lambda_n T} \\
    &\geq c \sum_{n > N} \frac{|a_{n-1}|^2}{\lambda_n} e^{-2\lambda_n T} \\
    &\geq c e^{-2C_\epsilon T}\sum_{n > N} \frac{|a_{n-1}|^2}{nq'(0)(1+\epsilon)} e^{-2nq'(0)(1+\epsilon)T} \\
    &\geq \frac{c e^{-2C_\epsilon T}}{q'(0)(1+\epsilon)}\sum_{n > N} \frac{|a_{n-1}|^2}{n + 1} e^{-2nq'(0)(1+\epsilon)T}
\end{align*}

Therefore, letting $p = \sum_{n \geq N} a_nz^n$ be a complex polynomial with a zero of order $N$ at zero, since the functions $z \mapsto z^n$ are orthogonal in $L^2(D(0,R),dm)$ with $dm$ the complex Lebesgue measure, the above computations give 
\begin{align}\label{eqn: lower bound left hand side complement rectangle}
    \int_{\Omega} |g(T,x,y)|^2 \ dx \ dy \geq C \|p\|^2_{L^2\left(D\left(0,e^{-q'(0)(1+\epsilon)T}\right)\right)},
\end{align}
for some constant $C > 0$ that may differ with the one given by the observability inequality. \\

We now treat the right-hand side of \eqref{eqn: obs form solution complement rectangle}, still following the proof of \cite[Lemma 3.7]{darde2023null}.\\

Recall that the Agmon distance is defined by \eqref{agmon distance gamma = 1} 
\begin{align} 
    d_{\operatorname{agm}} : x \in (-L,L) \mapsto \int_0^x q(s) \ ds.
\end{align}
Let $\epsilon \in (0,1)$, and define $\gamma_{t,x}(n)$ by 
\begin{align}\label{eqn: definition error term}
   \gamma_{t,x}(n - 1) := e^{-t(\lambda_n - q'(0)n)} v_n(x)e^{n d_{\operatorname{agm}}(x)(1-\epsilon)}.
\end{align}
Then, the solutions prescribed by \eqref{eqn: form solution complement rectangle} can be written  
\begin{align}\label{eqn: form solution complement rectangle error term}
    g(t,x,y) = \sum_{n > N} a_{n-1} \gamma_{t,x}(n - 1) e^{n(iy - q'(0)t - (1-\epsilon)d_{\operatorname{agm}}(x))}
\end{align}
Introduce the change of variable 
\begin{align}\label{eqn: change of variable complex}
    (x,z) = (x, e^{iy - q'(0)t - (1-\epsilon)d_{\operatorname{agm}}(x)})
\end{align}
which is a diffeomorphism from $(0,T) \times \omega$ onto its image that we denote by $\mathcal{P}$, and which satisfies $\mathcal{P} \subset \Omega_x \times U$ with (see Figure \ref{fig:diffeo_koenig}) 
\begin{align}\label{eqn: definition U change variable complex}
    U &= D\left(0,e^{-(1-\epsilon)d_{\operatorname{agm}}(-a)}\right) \cup D\left(0,e^{-(1-\epsilon)d_{\operatorname{agm}}(b)}\right) \cup \{z \in \mathbb{C}, |z| < 1, \operatorname{arg}(z) \notin I \} \nonumber \\[6pt]
    &= D\left(0,e^{-(1-\epsilon)\min(d_{\operatorname{agm}}(-a),d_{\operatorname{agm}}(b))}\right) \cup \{z \in \mathbb{C}, |z| < 1, \operatorname{arg}(z) \notin I \}.
\end{align}

Under the change of variable \eqref{eqn: change of variable complex}, the right-hand side of \eqref{eqn: obs form solution complement rectangle} becomes, writing $g$ as given in \eqref{eqn: form solution complement rectangle error term},
\begin{align}\label{eqn: right side obs complement rectangle}
    \int_0^T \int_\omega |g(t,x,y)|^2 \ dx \ dy \ dt = \frac{1}{q'(0)} \int_\mathcal{P} \left| \sum_{n > N} a_{n-1} \gamma_{t,x}(n - 1) z^n  \right|^2 \frac{1}{|z|^2} \ dx \ dm(z). 
\end{align}

Letting again $p$ be a complex polynomial with a zero of order $N$ at zero, it writes as $p(z) = \sum_{n \geq N} a_n z^n$. We see that the right-hand side above is almost an integral on $p$, and we will try to bound it from above by some $L^\infty$-norm of $p$. \\

Introduce the operator $\gamma_{t,x}$ acting on complex polynomials as
\begin{align}
    \gamma_{t,x}(z\partial_z)\left( \sum_n a_n z^n \right) = \sum_n a_n \gamma_{t,x}(n) z^n.
\end{align}
This operator is introduced in \cite[eq. (3.19)]{darde2023null} in the case $V = 0$, and its properties are given by \cite[Lemma 3.5]{darde2023null}. Therefore, \eqref{eqn: right side obs complement rectangle} is bounded as
\begin{align*}
    \int_0^T \int_{\tilde{\omega}} |g(t,x,y)|^2 \ dx \ dy \ dt  &= \frac{1}{q'(0)} \int_{\mathcal{P}} \left| \gamma_{t,x}(z\partial_z)(p)(z) \right|^2  \ dx \ dm(z) \\
    &\leq C \sup_{(t,x) \in (0,T) \times \Omega_x} \|\gamma_{t,x}(z\partial_z)(p)\|_{L^\infty(U)}.
\end{align*}

We have now arrived to the important point of the proof. We need to bound one last time the right-hand side above by
\begin{align}\label{eqn: admitted upper bound}
    \sup_{(t,x) \in (0,T) \times \Omega_x} \|\gamma_{t,x}(z\partial_z)(p)\|_{L^\infty(U)} \leq C \|p\|_{L^\infty(\mathcal{V})},
\end{align}
for some constant $C > 0$ independent of $(t,x)$, and $\mathcal{V}$ an arbitrary small neighborhood of $\overline{U}$.\\ 

Let us admit for now that this is possible and conclude the proof of Theorem \ref{th gamma = 1 euclidean koenig}, following the exact same lines as for the proof of \cite[Theorem 1.3, Theorem 3.1]{darde2023null}. We have seen that if observability \eqref{observability inequality Omega} holds in time $T>0$ from $\tilde{\omega}$, then from \eqref{eqn: lower bound left hand side complement rectangle} and the admitted upper bound \eqref{eqn: admitted upper bound}, there exists $C> 0$ such that for every complex polynomial $p$ with a zero of order $N$ at zero, $N$ large enough, we have
\begin{align}\label{inequality polynomials complex}
    \|p\|^2_{L^2\left(D\left(0,e^{-q'(0)(1+\epsilon)T}\right)\right)} \leq C \|p\|_{L^\infty(\mathcal{V})},
\end{align}
where we recall that $\mathcal{V}$ is an arbitrary small neighborhood of $\overline{U}$ defined in \eqref{eqn: definition U change variable complex}. For every $\epsilon > 0$, for every $T > 0 $ such that
\begin{align}
    T \leq \frac{1 - \epsilon}{q'(0)(1+\epsilon)}\min(d_{\operatorname{agm}}(-a),d_{\operatorname{agm}}(b)),
\end{align}
there exists $z_0 \in D(0,e^{-q'(0)(1+\epsilon)T}) \setminus \overline{U}$ (see Figure \ref{fig:diffeo_koenig}). Hence, by Runge's theorem, there exists a sequence of polynomials $(\Tilde{p}_k)_k$ that converges to $1/(z-z^0)$, uniformly on every compact subset of $\mathbb{C} \setminus z_0[1,+\infty)$. We then choose the sequence of polynomials $p_k = z^{N+1}\tilde{p}_k$ to disprove \eqref{inequality polynomials complex}. Indeed, as $k \rightarrow + \infty$, the left-hand side of \eqref{inequality polynomials complex} explodes, while the right-hand side remains uniformly bounded. \\

Now, since $\epsilon$ can be chosen arbitrary small, the result of Theorem \ref{th gamma = 1 euclidean koenig} follows. \\

\begin{figure}
    \centering
    \includegraphics[width=0.5\linewidth]{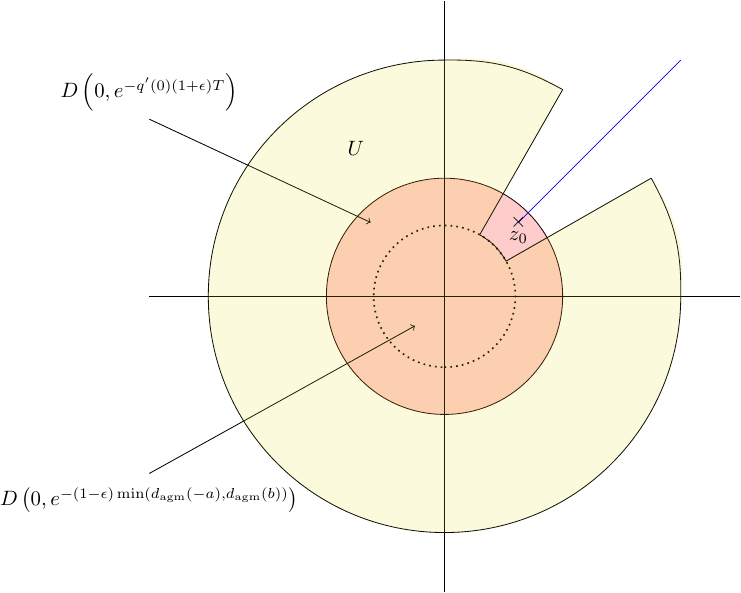}
    \caption{The yellow pacman shape is $U$. The disk in red is $D\left(0,e^{-(1-\epsilon)\min(d_{\operatorname{agm}}(-a),d_{\operatorname{agm}}(b))}\right) $, and it contains the dotted disk $D\left(0,e^{-q'(0)(1+\epsilon)T} \right)$. When $T$ is too small, there exists $z_0 \in D(0,e^{-q'(0)(1+\epsilon)T}) \setminus \overline{U}$. }
    \label{fig:diffeo_koenig}
\end{figure}

In particular, the sequence of complex numbers used in the solutions of the form \eqref{eqn: form solution complement rectangle} to disprove the observability inequality \eqref{eqn: obs form solution complement rectangle} are entirely determined by the sequence of polynomials $(\Tilde{p}_k)_k$ above. \\

To fully conclude on the present proof, recall that we are left with proving that for any complex polynomial with a zero of order $N$ at zero, with $N$ sufficiently large, we have 
\begin{align}\label{eqn: final estimate complex polynomials}
    \sup_{(t,x) \in (0,T) \times \Omega_x} \|\gamma_{t,x}(z\partial_z)(p)\|_{L^\infty(U)} \leq C \|p\|_{L^\infty(\mathcal{V})}.
\end{align}
This is \cite[Lemma 3.5]{darde2023null} which follows from \cite[Theorem 4.7]{darde2023null}, proved in \cite[Theorem 18]{CRMATH_2017__355_12_1215_0}, and the rest of this proof is dedicated to proving \eqref{eqn: final estimate complex polynomials}.\\

The idea is to observe that, at least not requiring for now the independence of $C>0$ on $(t,x)$ for now, this is possible if the function 
\begin{align*}
    z \mapsto \sum_{n\geq N} \gamma_{t,x}(n)z^n
\end{align*}
admits an holomorphic extension on $\mathbb{C} \setminus [1,+\infty)$. It has been shown in \cite[Theorem 18]{CRMATH_2017__355_12_1215_0} that a sufficient condition is that the error term $\gamma_{t,x}$ can be interpolated by some function in a certain class of holomorphic functions $\mathcal{S}(D)$ that we introduce below. \\

For the constant $C$ to be independent of $t,x$, we need the family of interpolations to be uniformly bounded in $\mathcal{S}(D)$ with respect to $t$ and $x$.  \\

From Section \ref{section spectral G_V complex}, for every $\theta \in [0,\frac{\pi}{2})$ there exists $r_\theta \geq 1$ such that Proposition \ref{proposition estimates resolvent} holds for any $\nu \in \Sigma_{\theta_0} \setminus D(0,r_{\theta_0})$. Following \cite[Definition 4.6]{darde2023null}, we define 
\begin{align*}
    D := \cup_{\theta \in [0,\pi/2)} \Sigma_{\theta} \setminus D(0,r_{\theta}),
\end{align*}
and $\mathcal{S}(D)$ to be the set of holomorphic function $f$ on $D$ such that for every $\theta \in [0,\pi/2),\delta > 0$ we have 
\begin{align*}
        \sup_{D \cap \Sigma_\theta} \left| f(z) e^{-\delta |z|}\right| < + \infty.
\end{align*}

Observe that $\mathcal{S}(D)$ is stable under multiplication. Hence, it is sufficient to show separately that both terms $e^{-t(\lambda_n - q'(0)n)}$ and $v_n(x)e^{n d_{\operatorname{agm}}(x)(1-\epsilon)}$ forming $\gamma_{t,x}(n)$ \eqref{eqn: definition error term} can be interpolated by an element of $\mathcal{S}(D)$. \\

Let us treat the first term $e^{-t(\lambda_n - q'(0)n)}$. Write $\lambda(z)$ to be the first eigenvalue of $-\partial_x^2 + z^2q(x)^2 + V(x)$ obtained in Section \ref{section spectral G_V complex}. With this notation, we have $\lambda_n = \lambda(n)$. Hence, the first term is interpolated by the function 
\begin{align}
    \mathcal{E}_t : z \mapsto e^{-t(\lambda(z) - q'(0)z)}.
\end{align}

Similarly, denoting $v_z = \Pi_z\left(z^{1/4}e^{-zq'(0)x^2/2}\right)$, the second term is interpolated by 
\begin{align}
    \mathcal{F}_x : z \mapsto v_z(x)e^{zd_{\operatorname{agm}}(x)(1-\epsilon)}.
\end{align}

From complex perturbation theory (see \textit{e.g.} \cite[Chapter 7.1]{Kato1995-xh}), both functions $\mathcal{E}_t$ and $\mathcal{F}_x$ are holomorphic.\\ 

Using Proposition \ref{proposition estimates resolvent}, we see that the interpolations $\mathcal{E}_t$ belong to $\mathcal{S}(D)$ and are uniformly bounded. Similarly, thanks to Proposition \ref{proposition agmon estimate complex} below, the interpolations $\mathcal{F}_x$ belong to $\mathcal{S}(D)$ and are uniformly bounded. For details on these last two facts, we refer the reader to \cite[Lemma 3.5]{darde2023null}. Thus, estimate \eqref{eqn: final estimate complex polynomials} is true, which concludes the proof.

\end{proof}

\begin{proposition}\label{proposition agmon estimate complex}
    Let $\theta_0 \in [0,\pi/2)$, and $\epsilon \in (0,\pi)$. There exists $C > 0$ such that for $\nu \in \Sigma_{\theta_0}$, $|\nu|$ large enough, we have
\begin{align}
    \int_I \left|v_\nu(x)e^{\nu(1-\epsilon)d_{\operatorname{agm}}(x)}\right|^2 \ dx &\leq C|\nu|, \\
    \|v_\nu(x)e^{\nu(1-\epsilon)d_{\operatorname{agm}}(x)}\|_{L^\infty(I)} &\leq C|\nu|.
\end{align}
\end{proposition}

\begin{proof}
We have the following Agmon equality from the exact same computations as in \cite[Proposition 4.3]{darde2023null},
\begin{align*}
    \|v_\nu'\|_{L^2(I)}^2 + \int_I (|\nu|^2q(x)(1-(1-\epsilon)^2) + V - \mu_\nu )|v_\nu(x)|^2 \ dx = 0,
\end{align*}
with 
\begin{align*}
    \mu_\nu := \frac{Re(e^{-i \arg(\nu)}\lambda_\nu)}{\cos(\arg(\nu))}.
\end{align*}
For some constant $C > 0$, for $|\nu|$ large enough, we have for some $c>0$ and every $x$ such that $d_{\operatorname{agm}(x)}|\nu| > C$,
\begin{align*}
    |\nu|^2q(x)(1-(1-\epsilon)^2) + V - \mu_\nu \geq c|\nu|,
\end{align*}
since in the limit, the potential, whether positive or negative, becomes negligible in front of the other terms. The proposition follows now from the exact same arguments as in \cite[Corollary 4.5]{darde2023null}.
\end{proof}

\begin{remark}\label{rmk: disprove control regular solutions}
    We observe from the above proof that Theorem \ref{th gamma = 1 euclidean koenig} is still true if we consider initial states in $D(G^s)$, for any $s > 0$. In this case, we have in the left-hand side of the observability inequality \eqref{observability inequality Omega} a $D(G^{-s})$-norm, and still an $L^2((0,T)\times \omega)$ norm in the right-hand side. Thus, we just have to bound from below the left-hand side. We have, following the computations of the beginning of the proof, 
    \begin{align*}
     \|g(T)\|_{D(G^{-s})}^2 \ dx \ dy &=  \sum_{n > N} \frac{|a_{n-1}|^2}{{\lambda_n}^{2s}} \| v_n \|_{L^2(\Omega_x)}^2 e^{-2\lambda_n T} \\
    &\geq \frac{C e^{-2C_\epsilon T}}{q'(0)^{2s}(1+\epsilon)^{2s}}\sum_{n > N} \frac{|a_{n-1}|^2}{n^{2s} + 1} e^{-2nq'(0)(1+\epsilon)T} \\
    &= \frac{C e^{-2C_\epsilon T}}{q'(0)^{2s}(1+\epsilon)^{2s}}\sum_{n \geq N} \frac{|a_{n}|^2}{(n+1)^{2s} + 1} e^{-2(n+1)q'(0)(1+\epsilon)T}. 
\end{align*}
Let $T > 0$ be an horizon time for which we know that \eqref{inequality polynomials complex} cannot hold. Denote by $(p_k)_k$ the sequence of polynomial refuting the inequality \eqref{inequality polynomials complex} in the above proof. That is, $p_k = z^{N+1}\tilde{p}_k = z^{N+1}\sum_n a_{n,k}z^n$, where $\tilde{p}_k$ converges uniformly on every compact subsets of $\mathbb{C} \setminus (z_0[1,+\infty))$ to $(z-z_0)^{-1}$, with $z_0$ chosen as in the proof of Theorem \ref{th gamma = 1 euclidean koenig}. We stress that in particular, we have $|z_0| < e^{-q'(0)(1+\epsilon)T}$. Thus, the solutions constructed from the proof of Theorem \ref{th gamma = 1 euclidean koenig} that refutes the observability \eqref{observability inequality Omega} are given by 
\begin{align*}
    g_k(t,x,y) = \sum_{n \geq N}a_{n-N,k}v_{n+1}(x)e^{-\lambda_{n+1}t + i(n+1)y}.
\end{align*}
Let us see that these $g_k$ also disprove the $D(G^{-s})-L^2$ observability. If observability would hold in time $T > 0$, we would obtain that there exists $C>0$, different from the observability constant, such that 
\begin{align}\label{eqn: Gs obs}
    \sum_{n > N} \frac{|a_{n-N-1,k}|^2}{n^{2s} + 1} e^{-2nq'(0)(1+\epsilon)T} \leq C \int_0^T \int_\omega |g_k(t,x,y)|^2 \ dx \ dy \ dt,
\end{align}
with the right-hand side uniformly bounded. We need to show that the left-hand side diverges as $k \rightarrow + \infty$. By the Cauchy formula, and by the uniform convergence of $(\tilde{p}_k)$ to $(z-z_0)^{-1}$, we have that for any $n \geq 0$, $a_{n,k}$ converges to $-1/z_0^{n+1}$. Thus, 
\begin{align*}
    \frac{|a_{n-N-1,k}|^2}{n^{2s} + 1} \rightarrow \frac{1}{n^{2s}} \left| \frac{1}{z_0}\right|^{2(n-N)}, \quad \text{as} \quad k \to + \infty.
\end{align*}
By Fatou's Lemma, we see that the left-hand term in \eqref{eqn: Gs obs} will diverge to infinity if $|z_0| < e^{-q'(0)(1+\epsilon)T}$, which is the case by our choice of horizon time $T > 0$. This concludes. 
\end{remark}

\begin{remark}
    It would be interesting to obtain this result when $r$ is not identically one, which seems to be a complicated task. Indeed, the holomorphic interpolation argument becomes a challenging issue. In this case, the first eigenvalue $\lambda_n$ becomes $\lambda_{\xi_n}$, where we recall that $\xi_n^2$ is the $n$-th eigenvalue of $\partial_y(r(y)^2\partial_y)$, and the associated eigenfunctions are now perturbations of sinusoidal, at least for large values of $n$. Ignoring the eigenfunctions issue as we believe it is not the most important one, for almost every $(t,x)$ we want to find an holomorphic function $F_{t,x} \in \mathcal{S}(D)$ such that $F_{t,x}(n) = \gamma_{t,x}(n)$. This first amounts to interpolate for example the term $e^{-t(\lambda_{\xi_n} - q'(0)n)}$ at each $n$. We saw that when $r = 1$, the interpolation is directly given by the same formula replacing $n$ by $z$, as the holomorphy is ensured by perturbation theory arguments. But when $r$ is not identically one, we must find $G$, holomorphic in the right half-plane (see \cite{CRMATH_2017__355_12_1215_0}), such that $G(n) = \xi_n$. Indeed, in this case, the composite function $z \mapsto G(z) \mapsto \lambda(G(z))$ interpolates correctly $\lambda_{\xi_n}$. The existence of an interpolation is not a complicated problem, but one must keep in mind that we have heavy restrictions on it since the interpolation of the error term must in the end belong to $\mathcal{S}(D)$.
\end{remark}

\section{Reduction to a local problem and proof of Theorem \ref{th nullcontr manifold}}\label{section proofs of theorem on manifolds}

In this section, we prove our main result, Theorem \ref{th nullcontr manifold}. Throughout this section, we set ourselves under the geometric setting prescribed by Section \ref{section setting}. Recall that, according to the statement of Theorem \ref{th nullcontr manifold}, when working under \ref{H1i}, our almost-Riemannian structure on $\mathcal{M}$ is of step $\gamma + 1 =2$, and under \ref{H1ii}, it is of step $\gamma + 1 \geq 2$. This will always be implied in our statements when we invoke \ref{H1i} or \ref{H1ii}. \\

We fix in system \eqref{control grushin system M} the control zone $\omega$  satisfying one of the conditions outlined in \ref{H1}. The singular set $\mathcal{Z}$ satisfies the assumptions of \ref{H0}, and, without loss of generalities, it is diffeomorphic to $\Omega_y = (0,\pi), \ \mathbb{S}^1$, or $\mathbb{R}$. The distance $\delta$ to the singularity satisfies \ref{H0'}. \\

We recall that we denote by $\mathcal{U}$ a tubular neighborhood constructed as in Section \ref{section setting}, under \ref{H1}, and such that \ref{H2} and \ref{H3} hold. That is, $\mathcal{U}$ is contained in $\mathbb{M} \setminus \omega$, and is diffeomorphic to $(-L,L) \times \Omega_y$ for some $0 < L \leq \min(L_0,R) < \min(d_{\operatorname{sR}}(\omega, \mathcal{Z}), \operatorname{inj}(\mathcal{Z}))$, where $L_0 > 0$ and $R>0$ are respectively such that \ref{H2}, \ref{H3} and \ref{H0'} hold. \\

We also recall that in $\mathcal{U}$, the sub-Riemannian structure is generated by two vector fields $X = \partial_x$ and $Y = q(x)r(y)\partial_y$, with $q$ and $r$ satisfying the conditions outlined in \ref{H2}.\\ 

To proceed to the proof, we need to characterize how the boundary of such a neighborhood can intersect the boundary of $\mathbb{M}$ under \ref{H1ii}. In the Lemma below and what follows, the notation $\simeq$ signifies \textit{diffeomorphic to}.

\begin{lemma}\label{lemma: bndry tub neighb}
     Let $\mathcal{U}$ be a tubular neighborhood as prescribed above under \ref{H1ii}, that is diffeomorphic to $(-L,L) \times \Omega_y$ for some $0 < L \leq \min(L_0,R) < \min(d_{\operatorname{sR}}(\omega, \mathcal{Z}), \operatorname{inj}(\mathcal{Z}))$. Then, for $L$ small enough,
    \begin{enumerate}[label = (\roman*)]
    \item if $\mathcal{Z} \simeq \mathbb{S}^1$ or $\mathbb{R}$, we have $\partial \mathcal{U} \cap \partial \mathcal{M} = \emptyset$,
    \item if $\mathcal{Z} \simeq (0,\pi)$, we have $\partial \mathcal{U} \cap \partial \mathcal{M} \simeq \left( [-L,L] \times \{0\} \right) \cup \left( [-L,L] \times \{\pi \} \right)$.
\end{enumerate}
\end{lemma}

\begin{proof}
    We need to show that for $L>0$ sufficiently small, we have $\mathcal{U} = \{ p \in \mathcal{M}, \delta(p) < L\}$. Indeed, assume that this holds.\\
    
    If $\mathcal{Z} \simeq \mathbb{S}^1$, then $\partial \mathcal{U} = \{ p \in \mathcal{M}, \delta(p) = L\} \simeq \left(\{-L\} \times \mathbb{S}^1 \right) \cup \left( \{L\} \times \mathbb{S}^1 \right)$. In this case, since $\mathcal{Z} \subset \operatorname{Int}(\mathcal{M})$, the lemma directly follows as $L$ can be chosen small enough so that $\mathcal{U} \subset \operatorname{Int}(\mathcal{M})$. The same idea goes for $\mathcal{Z} \simeq \mathbb{R}$. \\
    
    If $\mathcal{Z} \simeq (0,\pi)$, let $p = (0,0) \in \partial\mathcal{U} \cap \partial\mathcal{Z} \cap \partial\mathcal{M}$ by \ref{H0}, and consider, in coordinates, the curve $c : t \in [0, L] \mapsto c(t) = (0,t)$ which belongs to $\partial \mathcal{U}$. Assume that $c$ has exited $\partial \mathcal{M}$ at some time $t_0 \in (0,L)$. Then, there exists an open ball around $c(t_0)$ of radius $\epsilon > 0$ sufficiently small such that $B(c(t_0),\epsilon) \subset \operatorname{Int}(\mathcal{M})$ and $t_0 + \epsilon < L$, and a point $p' \in B(c(t_0),\epsilon) \setminus \mathcal{U}$. By continuity of $\delta$ from \ref{H0'}, we have $\delta(p') \leq t_0 + \epsilon < L$. Hence $p' \in \{ p \in \mathcal{M}, \delta(p) < L\}$ which is a contradiction. The other parts of $\partial \mathcal{U}$ can be treated the same way. Taking $L$ small enough directly concludes the proof. \\

    It remains to prove that for $L \in (0,\min(d_{\operatorname{sR}}(\omega, \mathcal{Z}), \operatorname{inj}(\mathcal{Z}))$ we have $\mathcal{U} = \{ p \in \mathcal{M}, \delta(p) < L\}$. \\ 
    
    The set equality is proved in the case $\mathcal{Z} \simeq \mathbb{S}^1$ in \cite[Proposition 3.1]{Franceschi2020-gs}, \cite[Theorem 3.7]{rossi2022relative}. When $\mathcal{Z} \simeq \mathbb{R}$, it still holds thanks to the injectivity radius assumption in \ref{H0} (see \cite[footnote p.98]{Franceschi2020-gs}). \\
    
    Let us focus on the remaining case. We obviously have the inclusion $\mathcal{U} \subset \{ p \in \mathcal{M}, \delta(p) < L\}$, so let us prove the reciprocal one. Let $p_0 \in \{q \in \mathcal{M}, \delta(q) < L\} \setminus \mathcal{U}$. Let $q \in \mathcal{Z}$ be the unique point in $\mathcal{Z}$ such that $\delta(p_0) = d_{\operatorname{sR}}(p_0,q)$. Then, there exists $c_0 : [0,1] \rightarrow \mathcal{M}$ a minimizing geodesic such that $c_0(0) = q$ and $c_0(1) = p_0$. Assume that there exists $p_1 \in \mathcal{U}$ in the same connected component of $\{q \in \mathcal{M}, \delta(q) < L\}$ as $p_0$ such that $\delta(p_1) = d_{\operatorname{sR}}(p_0,q) = \delta(p_0)$. We know that $p_1$ exists  since it is either $p_1 = (q,\delta(p_0))$ or $p_1 = (q, -\delta(p_0))$ depending on where $p_0$ is situated. There exists a minimizing normal geodesic $c_1$ such that $c_1(0) = q$ and $c_1(1) = p_1$. Observe that the minimizing geodesics $c_i$, $i=0,1$, must satisfy $c_i'(t) = \alpha(t)\nabla\delta(c_i(t))$, where $\alpha : [0,1] \rightarrow \mathbb{R}^+$ is smooth. Since we assumed smoothness of $\delta$ in \ref{H0'}, we must have that there exists $\beta > 0$ such that $c_0'(0) = \beta c_1'(0)$. Since $c_0$ and $c_1$ are both integral curves of $\nabla \delta$ starting from the same point, we must have that there exists a smooth function $\beta : [0,1] \rightarrow \mathbb{R}^+$ such that $\beta(0) = \beta$ and $c_1'(t) = \beta(t)c_0'(t)$. That is, $c_1$ is a reparametrization of $c_0$, and hence $c_0 (1) = c_1(1)$. Thus, $p_0 = p_1$ and the inclusion $\{ p \in \mathcal{M}, \delta(p) < L\} \subset \mathcal{U}$ follows. This concludes the proof.
\end{proof}

\begin{remark}
    What the above proof says, is that assuming smoothness of $\delta$ as in \ref{H0'}, we only have a unique way of exiting $\mathcal{Z}$ in each side of it, by means of a minimizing geodesic (minimizing the distance to $\mathcal{Z}$), which is a normal one, from any point of the singularity. This geodesic must be an integral curve for $\nabla \delta$. If we exit from $q \in \operatorname{Int}(\mathcal{Z})$, this geodesic exists, is unique, and must be normal (see \cite[Proposition 2.7]{Franceschi2020-gs} or \cite[Proposition 2.2]{bossio2024tubes}). However, this is not true in all generalities from $q \in \partial \mathcal{Z}$, as it may be done by means of different normal geodesics, or abnormal ones. However, the existence of a normal geodesic is ensured. However, the smoothness assumption from \ref{H0'} forces all possibilities of leaving $\mathcal{Z}$ from $q$ to be along the same path. That is, even if $q \in \partial \mathcal{Z}$, this geodesic must be unique (modulo reparametrization) and normal. This in turns forces $\mathcal{U} = \{ p \in \mathcal{M}, \delta(p) < L\}$. 
\end{remark}

\subsection{Reduction process}\label{section reduction process}

We fix from now on $\mathcal{U}$ to be a tubular neighborhood as prescribed by Lemma \ref{lemma: bndry tub neighb} under \ref{H1ii}, which we recall is diffeomorphic to $(-L,L) \times \Omega_y$. Under \ref{H1i}, we also fix $\mathcal{U}$ similarly to be a tubular neighborhood of a point $p \in \mathcal{Z}$ such that $\partial\mathcal{U} \subset \operatorname{Int}(\mathcal{M})$, which is always possible. Under \ref{H1i}, $\mathcal{U} \simeq \Omega = (-L,L) \times (0,\pi)$ and $\partial \mathcal{U} \simeq \partial\Omega$. \\

We show in this subsection that null-controllability of system \eqref{control grushin system M} implies the internal null-controllability of initial states in the form domain $D \left(\Delta_\mathcal{U}^{1/2} \right)$ of system 
\begin{align}\label{control grushin system local}
        \left\{ \begin{array}{lcll}
        \partial_t f - \Delta_\mathcal{U} f & = & \mathbf{1}_{\omega^\epsilon}(x,y) u(t,x,y), & (t,x,y) \in (0,T) \times \mathcal{U},\\
        f(t,x,y) & = & 0, & (t,x,y) \in (0,T) \times \partial \mathcal{U}, \\
        f(0,x,y) & = & f_0(x,y), & (x,y) \in \mathcal{U}.
        \end{array} \right.
    \end{align}
For simplicity of the presentation, we shall denote the elements of $\mathcal{U}$ by $(x,y)$ with $x \in (-L,L)$ and $y \in \Omega_y$.  

\begin{proposition}\label{proposition internal M to U}
    Under the fixed setting previously prescribed, assume that system \eqref{control grushin system M} is null-controllable in time $T > 0$ from $\omega$ that satisfies \ref{H1i}. Let $\epsilon > 0$, and define 
    \begin{align*}
        \omega^\epsilon = \left\{ (x,y) \in \mathcal{U}, \quad x \in (-L,-L+\epsilon) \cup (L-\epsilon,L), \quad y \in (0,\epsilon) \cup (\pi - \epsilon, \pi) \right\}.
    \end{align*}
    Then, for every $\epsilon > 0$, for every $f_0 \in D \left(\Delta_\mathcal{U}^{1/2} \right)$, there exists a control $u \in L^2((0,T) \times \omega^\epsilon)$ such that the associated solution of system \eqref{control grushin system local} satisfies $f(T) = 0$ in $L^2(\mathcal{U})$. \\

    Moreover, if the control zone $\omega$ in system \eqref{control grushin system M} satisfies \ref{H1ii}, then the same conclusion holds with  
    \begin{align*}
        \omega^\epsilon = \left\{ (x,y) \in \mathcal{U}, \quad x \in (-L,-L+\epsilon) \cup (L-\epsilon,L) \right\},
    \end{align*}
    for every $\epsilon \in (0,L)$.
\end{proposition}

We recall that when working under \ref{H1i}, our structure is of step $\gamma + 1 =2$, and under \ref{H1ii}, it is of step $\gamma + 1 \geq 2$. In the second part of the above Proposition, if $\gamma + 1 > 2$, and $\epsilon$ is such that $\omega^\epsilon = \mathcal{U}$, the conclusion of our proof of Proposition \ref{proposition internal M to U} is no longer true. Using our approach via cutoff arguments, when the control zone $\omega^\epsilon$ intersects the singularity, one needs to be careful on the regularity of the initial states in system \eqref{control grushin system local} that can be steered to $0$ by means of $L^2$-controls induced by the assumption that system \eqref{control grushin system M} is null-controllable. For a structure of step $k$, one would at least be required to consider initial states in $D\left( \Delta_\mathcal{U}^{\frac{k-1}{2}}\right)$. However, in the proof below, some extension arguments would need to be carefully checked. Thus, under \ref{H1ii}, for simplicity, we restrict $\epsilon > 0$ to be small so that $\omega^\epsilon$ is located within the elliptic region, and so that our argument works when $\gamma + 1 > 2$.     

\begin{proof}
Let $T>0$. Assume that we are under \ref{H1i}, and that system \eqref{control grushin system M} is null-controllable in time $T>0$ from $\omega$. The following arguments hold due to the analyticity of the semigroup associated to our operators $\Delta$ and $\Delta_\mathcal{U}$. \\

Let $f_0 \in D \left( \Delta_{\mathcal{U}}^{1/2} \right)$. Extend it by $0$ outside of $\Omega$. Then, the function $\mathbf{1}_\mathcal{U}f_0$ belongs to the form domain $D\left( \Delta^{1/2}\right)$. Indeed, this can be seen by approximating $f_0$ by smooth compactly supported functions in $\mathcal{U}$, that are extended outside of $\mathcal{U}$ by zero. Then, as $D\left( \Delta^{1/2}\right)$ is the completion of $C_c^\infty(\mathcal{M})$ with respect to the horizontal Sobolev norm 
\begin{align*}
    \|u\|_{D(\Delta^{1/2})}^2 = \int_\mathcal{M} \sum_{i=1}^m |X_i u(p)|^2 \ d\mu(p), 
\end{align*}
where the family $\left\{ X_1,...X_m \right\}$ is a generating frame for the sub-Riemannian structure, this concludes.
\\

Now, by assumption, there exists a control $v \in L^2((0,T) \times \omega, dtd\mu)$ such that the solution $f_{\operatorname{man}}$ of system \eqref{control grushin system M} with initial state $\mathbf{1}_\mathcal{U} f_0$ satisfies $f_{\operatorname{man}}(T) = 0$. Denote now by $f_{\operatorname{adj}}$ the solution of system \eqref{control grushin system local} with $u = 0$ and with initial state $f_0$. We set
\begin{align*}
    f_{\operatorname{tub}} := \eta \theta f_{\operatorname{adj}} + (1- \theta)f_{\operatorname{man}},
\end{align*}
where $\eta \in C^\infty([0,T])$ and $\theta = \theta_1\theta_2 \in C^\infty(\mathcal{U})$ are defined by 
\begin{align*}
    \left\{ \begin{array}{lllr}
        \theta_1(x) & = & 1, & x \in [-L, -L + \frac{\epsilon}{2}] \cup [L - \frac{\epsilon}{2}, L], \\  
        \theta_1(x) & = & 0, & x \in [-L + \epsilon,L - \epsilon], \\
        \theta_1(x) & \in & [0,1], & x \in [-L, L],
    \end{array} \right.
\quad
    \left\{ \begin{array}{lllr}
        \theta_2(x) & = & 1, & x \in [0, \frac{\epsilon}{2}] \cup [\pi - \frac{\epsilon}{2}, \pi], \\  
        \theta_2(x) & = & 0, & x \in [\epsilon, \pi - \epsilon], \\
        \theta_2(x) & \in & [0,1], & x \in [0,\pi],
    \end{array} \right.
\end{align*}

and

\begin{align}
    \left\{ \begin{array}{lcll}
        \eta(t) & = & 1, & t \in [0,\frac{T}{3}],  \\  
        \eta(t) & = & 0, & t \in [\frac{2T}{3},T], \\
        \eta(t) & \in & [0,1], & t \in [0,T].
    \end{array} \right.
\end{align}

We observe that in $\mathcal{U}$, $f_{\operatorname{tub}}(0) = f_0$, $f_{\operatorname{tub}}(T) = 0$, and $f_{\operatorname{tub}}$ values zero at the boundary. Moreover, computations in coordinates using \eqref{definition Delta} show that applying the parabolic operator $\partial_t - \Delta_\mathcal{U}$ to $f_{\operatorname{tub}}$ produces a source term $h$ localized in $\omega^\epsilon$, that will depend on $f_{\operatorname{adj}}$, $f_{\operatorname{man}}$ and their first derivatives with respect to $x$ and $y$.\\

To conclude on the proof, we need to make sure that $h \in L^2((0,T) \times \omega^\epsilon)$. A sufficient condition for this is that $f_{\operatorname{adj}}$ and $f_{\operatorname{man}}$ both belong to $L^2 \left((0,T);H^1(\mathcal{U}) \right)$.\\

Let us first observe that since $f_0 \in D \left( \Delta_{\mathcal{U}}^{1/2} \right)$, then $f_{\operatorname{adj}} \in L^2((0,T),D \left(\Delta_\mathcal{U} \right))$. Subelliptic estimates (see e.g. \cite[Theorem 1.5 eq. (1.6)]{Laurent2022-ca}), stating that $D(\Delta_\mathcal{U}) \subset H^{2/k}(\mathcal{U})$ where $k$ is the step of the structure, then infer that for a.e. $t \in [0,T]$, $f_{\operatorname{adj}}(t) \in H^1(\mathcal{U})$, since, as we work under \ref{H1i}, the step of the structure is $2$. Thus, we indeed have $f_{\operatorname{adj}} \in L^2 \left((0,T);H^1(\mathcal{U}) \right)$. \\

Now, to treat $f_{\operatorname{man}}$, we write it using the Duhamel formula
\begin{align*}
    f_{\operatorname{man}}(t) = e^{-t\Delta}\mathbf{1}_\mathcal{U}f_0 + \int_0^T e^{-(t-s)\Delta}\mathbf{1}_\omega v(s) \ ds. 
\end{align*}
We have that $e^{- \cdot \Delta}\mathbf{1}_\mathcal{U}f_0 \in L^2((0,T),D \left(\Delta\right))$, and thus belongs to $L^2 \left((0,T);H^1(\mathcal{U}) \right)$ using the subelliptic estimate \cite[Theorem 1.5 eq. (1.6)]{Laurent2022-ca} as above. For the non-homogeneous part, that we denote by $F_{\operatorname{man}}$, we use maximal $L^2$-regularity of our operator $\Delta$, which gives us the existence of a constant $C>0$ such that 
\begin{align*}
    \|\Delta F_{\operatorname{man}}\|_{L^2\left((0,T),L^2(\mathcal{M},d\mu) \right)}^2 \leq C  \|\mathbf{1}_\omega v\|_{L^2\left((0,T),L^2(\mathcal{M},d\mu) \right)}^2 < \infty.
\end{align*}
That is, $F_{\operatorname{man}} \in L^2 \left( (0,T);D(\Delta)\right)$. Thus, using once again the subelliptic estimate \cite[Theorem 1.5 eq. (1.6)]{Laurent2022-ca}, $F_{\operatorname{man}} \in L^2 \left( (0,T);H^1(\mathcal{M}) \right) \subset  L^2 \left( (0,T);H^1(\mathcal{U}) \right)$. This concludes on the fact that $f_{\operatorname{man}} \in  L^2 \left( (0,T);H^1(\mathcal{U}) \right)$. \\

The $L^2$-control $u$ in system \eqref{control grushin system local} that steers $f_0$ to $0$ in time $T$ is therefore given by $h$, and the associated solution is $f_{\operatorname{tub}}$.\\

Now, under \ref{H1ii}, we follow the same proof as above until the use
of the cutoffs, at which point we simply take $\theta = \theta_1$. The resulting non-homogeneous term $h$ acting like the control is now located, for $\epsilon \in (0,L)$, within the elliptic region. Thus, by local regularity of the solutions in the elliptic region, we indeed have $h \in L^2((0,T) \times \omega^\epsilon)$. Finally, thanks to Lemma \ref{lemma: bndry tub neighb}, the Dirichlet boundary conditions are also satisfied. This fully concludes the proof.  
\end{proof}

\subsection{Proof of Theorem \ref{th nullcontr manifold}}\label{section proof of Theorem nullcontr manifold}

We can now proceed to conclude on the proof of Theorem \ref{th nullcontr manifold}. Both considerations in the theorem are treated the same way.

\begin{proof}[Proof of Theorem \ref{th nullcontr manifold}]

We illustrate the proof in Figure \ref{fig reduction system}. We assumed that system \eqref{control grushin system M} is null-controllable in time $T > 0$ from $\omega$ in the beginning of Section \ref{section proofs of theorem on manifolds}. \\

Let us recall $\mathcal{U}$ is fixed at the beginning of Section \ref{section proofs of theorem on manifolds}. Under both considerations outlined in \ref{H1}, $\mathcal{U}$ is diffeomorphic to, via lets say $\Phi$, $\Omega = (-L,L) \times \Omega_y$ with $L>0$ given by Lemma \ref{lemma: bndry tub neighb}. \\

Moreover, without loss of generality, we assumed that  
\begin{align*}
    \Omega_y = \left\{ \begin{array}{ll}
        (0,\pi)& \text{under \ref{H1i}},  \\
         (0,\pi), \ \mathbb{S}^1, \text{ or }\mathbb{R} & \text{under \ref{H1ii}}. 
    \end{array} \right.
\end{align*}

The diffeomorphism $\Phi : \mathcal{U} \rightarrow \Omega $ induces a unitary transformation 
\begin{align*}
    \begin{array}{lccc}
         S : & L^2( \mathcal{U}, \mu) & \rightarrow & L^2(\Omega, h(x)dxdy)  \\
         & f & \mapsto & Tf = f \circ \Phi,
    \end{array}
\end{align*}
and we have 
\begin{align}
    G := S \Delta_\mathcal{U} S^{-1} = \frac{1}{h(x)}\partial_x(h(x) \partial_x ) + q(x)^2\partial_y( r(y)^2 \partial_y ).
\end{align}

Thus, by Proposition \ref{proposition internal M to U}, and the above unitary transformation, we have internal null-controllability in time $T$ of initial states $f_0 \in D\left( G^{1/2} \right)$ of system 
\begin{align}\label{control system internal omega no change variable}
    \left\{ \begin{array}{lcll}
     \partial_t f - (\frac{1}{h(x)}\partial_x(h(x)\partial_x) + q(x)^2\partial_y(r(y)^2\partial_y)) f & = & u(t,x,y)\mathbf{1}_{\omega^\epsilon}(x,y), & (t,x,y) \in (0,T) \times \Omega,\\
     f(t,p) & = & 0, & (t,p) \in (0,T) \times \partial \Omega, \\
     f(0,p) & = & f_0(p), & p \in \Omega,
    \end{array} \right.
\end{align}
posed in $L^2(\Omega, h(x)dxdy)$, and with the $L^2$-control $u$ supported on $\omega^\epsilon$ defined by, for $\epsilon > 0$ arbitrary small, 
\begin{enumerate}[label=(\roman*)]
    \item $\omega^\epsilon = \left\{ (x,y) \in \Omega, \ x \in (-L,-L+\epsilon) \cup (L - \epsilon,L) \right\} \cup \left\{ (x,y) \in \Omega, y \in (0,\epsilon) \cup (\pi - \epsilon, \pi) \right\}$, under \ref{H1i},
    \item $\left\{ (x,y) \in \Omega, \ x \in (-L,-L+\epsilon) \cup (L - \epsilon,L) \right\}$, under \ref{H1ii}.
\end{enumerate}

Consider now the unitary transformation 
\begin{align}
\begin{array}{lccl}
T : & L^2(h(x)dxdy) & \longrightarrow & L^2(dxdy) \\
& f &  \longmapsto & \sqrt{h}f.
\end{array}
\end{align}
Then, setting $G_V := TGT^{-1} = \partial_x^2f + q(x)^2\partial_y(r(y)^2 \partial_y f) + V(x)f$, the observability inequality associated to the problem of internal null-controllability in time $T  > 0$ of initial states in $D\left( G^{1/2} \right)$ of system \eqref{control system internal omega no change variable}, is equivalent to the observability inequality \eqref{observability inequality Omega}, associated to the problem of internal null-controllability in time $T  > 0$ of initial states in $D\left( G_V^{1/2} \right)$ of system \eqref{control system change variable}, with $V(x) = \frac{\partial_x^2(\sqrt{h})}{\sqrt{h}}$. \\

Therefore, the non null-controllability results presented in Theorem \ref{th nullcontr manifold} are inherited from both Theorems \ref{th gamma = 1 euclidean koenig} and \ref{th gamma >= 1 euclidean bounded}.

\begin{enumerate}
    \setlength\itemsep{1em}
    \item Under \ref{H1i}, and $\gamma = 1$, from Theorem \ref{th gamma = 1 euclidean koenig} we have that for every $\epsilon > 0$ arbitrary small, for any time
    \begin{align*}
        T < \frac{1}{q'(0)}\min \{ d_{\operatorname{agm}}(-L+\epsilon), d_{\operatorname{agm}}(L-\epsilon)\},
    \end{align*}
    there exists $f_0 \in D\left( G_V^{1/2} \right)$ such that the associated solution of system \eqref{control system change variable} cannot be steered to $0$ in time $T$. As $\epsilon$ can be taken arbitrary small, we get for system \eqref{control grushin system M} that
    \begin{align*}
        T(\omega) \geq  \frac{1}{q'(0)}\min \{ d_{\operatorname{agm}}(-L), d_{\operatorname{agm}}(L) \}.
    \end{align*}
    
    \item Under \ref{H1ii}, by Theorem \ref{th gamma >= 1 euclidean bounded}, when $\gamma > 1$ we obviously never have null-controllability, since system \eqref{control system change variable} is never null-controllable, and for $\gamma = 1$, we obtain the same lower bound for $T(\omega)$ as above.
\end{enumerate}
\end{proof}

\begin{figure}[ht]
\begin{subfigure}[h]{0.45\linewidth}
\includegraphics[width=\linewidth]{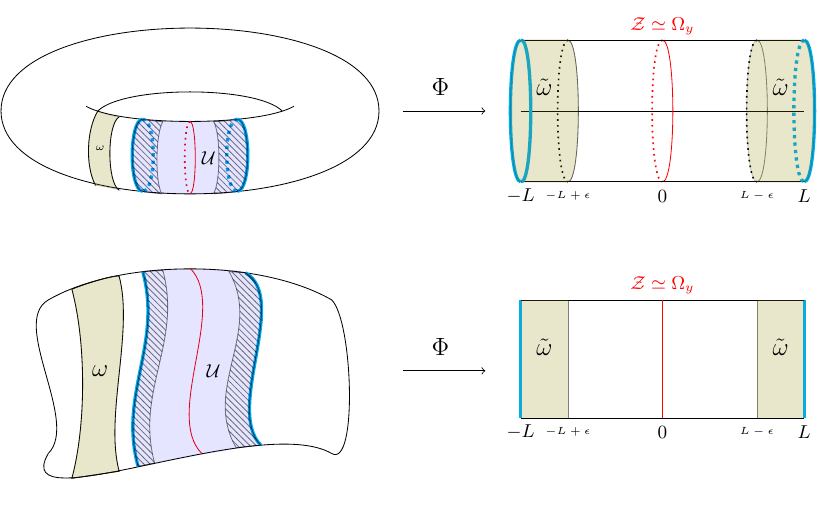}
\caption{The reduction process under \ref{H1ii}.}
\end{subfigure}
\quad
\begin{subfigure}[h]{0.45\linewidth}
\includegraphics[width=\linewidth]{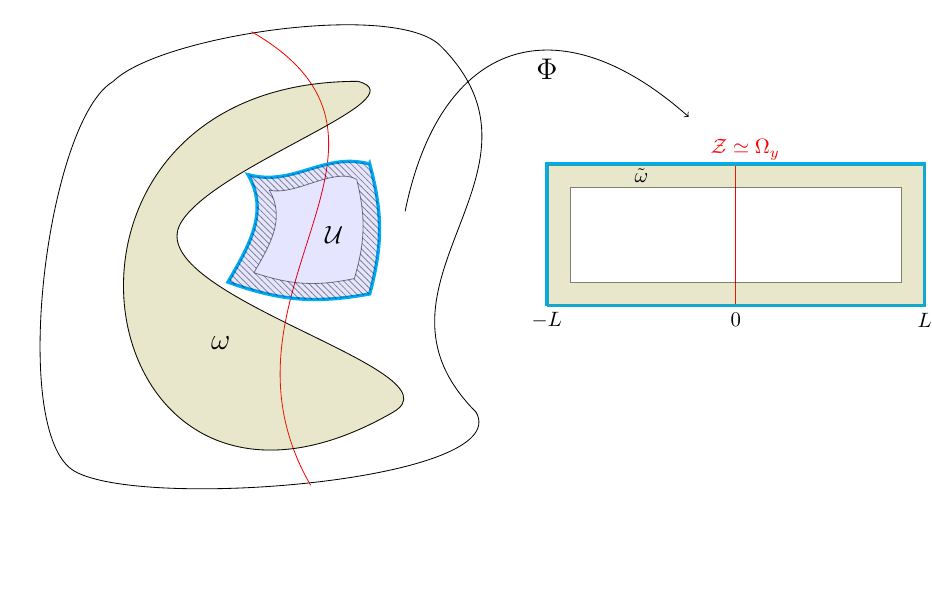}
\caption{The reduction process under \ref{H1i}.}
\end{subfigure}%
\caption{The reduction process under \ref{H1}. The singularity is in red, the control zones in olive, and the tubular neighborhood $\mathcal{U}$ in light blue. The blue paths correspond to $\partial \mathcal{U} \setminus \partial \mathbb{M}$. The dashed part of $\mathcal{U}$ correspond to where the control is supported after applying Proposition \ref{proposition internal M to U}. The arrows represent the passage into coordinates representation, with in olive on the tensorized domains the support of the control for the problem of internal controllability, which is the image by $\Phi$ of the dashed parts.}
\label{fig reduction system}
\end{figure}

\section{Applications, comments, and open problems}\label{section comments}

\textbf{The Grushin operator without warped product assumption.} Consider $\mathcal{M} = (-1,1) \times (0,\pi)$ with the Lebesgue measure. Set $X= \partial_x$, $Y=\Tilde{q}(x,y)\partial$. The sub-Laplacian writes $\Delta = \partial_x^2 + \partial_y(\Tilde{q}(x,y)^2\partial_y)$. Although there is no result for the heat equation associated to this operator, Theorem \ref{th nullcontr manifold} says that to obtain a negative result, it is sufficient that $\Tilde{q}$ locally writes as $\Tilde{q}(x,y) = q(x)r(y)$. Having a result for $\Delta$, even in the particular case of the rectangle, is an open problem. The usual Fourier techniques are not valid in this case.\\

\textbf{The Grushin sphere.} We retrieve with the strategy of Theorem \ref{th nullcontr manifold} the negative result of Tamekue \cite{tamekue2022null}. Namely, consider $\mathcal{M} = \mathbb{S}^2$, endowed with the sub-Riemannian structure generated by the vector fields $X = -z\partial_y + y\partial_z$,  $Y = -z \partial_x + x\partial_z$. Observe that they are linearly independent outside $\{z = 0\}$, but  $\{X,Y\}$ is bracket-generating since $[X,Y] = -y\partial_x + x\partial_y$. Moreover, \ref{H0} is satisfied. We also endow $\mathbb{S}^2$ with the restricted Lebesgue measure. Consider system \eqref{control grushin system M} with $\omega$ consisting of two symmetric, with respect to $\{z = 0\}$, horizontal crowns as in \cite[Fig. 1]{tamekue2022null}. For example $\omega = \{ (x,y,z) \in \mathbb{S}^2, |z| > a\}$ for some $a > 0$. Then, reducing the problem in $\mathbb{S}^2 \setminus \omega$, in spherical coordinates, the restricition of the Lebesgue measure writes $\cos(x)dxdy$, and we have that $\Delta = \frac{1}{\cos(x)}\partial_x(\cos(x)\partial_x) + \tan(x)^2\partial_y^2$ on $(-a,a) \times \mathbb{S}^1$. Then, by parity of the function $\tan$, our result shows that $T(\omega) \geq \frac{1}{q'(0)}\int_0^a \tan(s) \ ds = \ln\left(\frac{1}{\cos(a)}\right)$, which coincides with the negative result in \cite{tamekue2022null}. This strategy would also hold if the two components of the control zone are not at the same nonnegative distance of the singularity. \\

\textbf{The Grushin equation on the real line and lack of optimality.} Consider the Grushin equation on $\mathcal{M} = \mathbb{R} \times (0,\pi)$, with the Lebesgue measure, controlled on vertical strips. Theorem \ref{th nullcontr manifold} says that we can reduce the analysis of the negative results to the case $(-L,L) \times (0, \pi)$. We can emphasize that when $\lim_{|x| \rightarrow + \infty} |q(x)| = + \infty$, and $q$ satisfies \ref{H2}, the strategy proposed in the case where $\Omega_x$ is bounded extends to $\mathbb{R}$, and also if we perturb the equation by a potential $V \in L_{loc}^\infty(\mathbb{R})$ bounded below. The strategy of Theorem \ref{th gamma >= 1 euclidean bounded} extended to $\mathbb{R}$ and Theorem \ref{th nullcontr manifold} both give the same results. Nonetheless, it is important to notice that in both cases we obtain a lower bound for the minimal time of null-controllability, even when the equation is never null-controllable (for example when we do not control in one of the connected component of $\mathcal{M} \setminus \mathcal{Z}$).  \\

\textbf{Non-symmetry of the tubular neighborhood.} While Theorems \ref{th gamma = 1 euclidean koenig} and \ref{th gamma >= 1 euclidean bounded} assume that $\Omega_x$ is symmetric with respect to zero, they remain valid when substituting $(-L,L)$ with any other interval containing $0$ in its interior, up to some minor adjustments, in particular for Theorem \ref{th gamma >= 1 euclidean bounded}. For Theorem \ref{th gamma = 1 euclidean koenig}, the result comes from \cite{darde2023null} where the symmetry is never assumed. For Theorem \ref{th gamma >= 1 euclidean bounded}, the argument for exponential decay of the classical eigenfunctions in Proposition \ref{proposition supersolution grushin model} is no longer valid on the whole interval. That is, by lack of parity, Remark \ref{remark estimates supersolutions model} does not hold. One can either reapply the proof to negative $x$, or obtain an exponential decay argument as in Proposition \ref{proposition agmon decay grushin general}. In any case, the asymptotic on the eigenvalues still hold, and so does Theorem \ref{th gamma >= 1 euclidean bounded}. As a consequence, we can choose $\mathcal{U}$ to be diffeomorphic to $(-L_-,L_+) \times \Omega_y$. \\

\textbf{Positive results on $\mathcal{M}$}. Although we only focus on negative results, positive results are achievable in our settings when $\gamma = 1$. Indeed, if a positive result is achieved in $\mathcal{U}$ in time $T>0$ from $\omega$, it is possible to extend it to the whole manifold through cutoffs arguments. However, the resulting control zone $\omega' \subset \mathcal{M}$ for the system posed on $\mathcal{M}$ will have to be taken accordingly to the achievable positive results for the system posed on $\mathcal{M} \setminus \mathcal{U}$. For instance, under \ref{H1i}, to the best of our knowledge, we can only state that the heat system posed on $\mathcal{M} \setminus \mathcal{U}$ is null-controllable in any time $T > 0$ from $\mathcal{M} \setminus \mathcal{U}$ by dissipativity. Under \ref{H1ii} however, we can take advantage of the uniform parabolicity of the heat outside $\mathcal{U}$. \\

Nevertheless, to obtain null-controllability within $\mathcal{U}$, we must also require that $\inf_{\mathcal{U}}q'(x) > 0$, which is a stronger assumption than \ref{H2} (see \cite{beauchard2020minimal}). Proving positive results for the Grushin equation on rectangular domains without the monotony assumption on $q$ is still an open problem when the control zone does not act on both side of the singularity. \\

\textbf{Complement of a horizontal strip.} The analogue on manifolds of controlling in the complement of a horizontal strip on a rectangular domain, is to say that there exists a normal geodesic to the singularity that never intersects the control zone. Although we expect to never have null-controllability of any $\gamma \geq 1$, Theorem \ref{th nullcontr manifold} does not cover this case correctly. It only gives a lower bound for $T(\omega)$. It is also important to say that a negative result when $\gamma > 1$ in this setting on the rectangle has never been obtained, and is an open problem. If such a result is obtained, Theorem \ref{th nullcontr manifold glob} would hold under \ref{H1i}. 

\section*{Acknowledgment}
    The author would like to express his gratitude to his PhD advisors, Pierre Lissy and Dario Prandi, for their guidance, insightful discussions, and feedbacks throughout this work. The author also wishes to extend his thanks to Armand Koenig for the discussions regarding his series of papers on the non-observability of the Grushin equation, as well as for the exchanges concerning their generalization, and to anonymous reviewers whose comments greatly contributed to improving a previous version of this paper.

\printbibliography

\end{document}